\documentclass[a4paper, 11pt]{article}

\usepackage[utf8]{inputenc}
\usepackage[T1]{fontenc}
\usepackage[english]{babel}
\usepackage{amsmath}
\usepackage{amsthm}
\usepackage{amsfonts}
\usepackage{amssymb}
\usepackage{amstext}
\usepackage{dsfont}
\usepackage{graphicx}
\usepackage{color}
\usepackage[colorlinks]{hyperref}
\usepackage{epigraph}
\usepackage[left=3.5cm,top=2cm,bottom=3.5cm,right=3cm,nohead,nofoot]{geometry}
\usepackage{tikz}
\usepackage{tkz-graph}
\usepackage{tkz-berge}
\usetikzlibrary{arrows,shapes,matrix}
\graphicspath{{Figures/}{.}}
\allowdisplaybreaks[4]
\usepackage{enumitem}

\title{Mod-Gaussian convergence for random determinants and random characteristic polynomials}
\author{M. Dal Borgo, E. Hovhannisyan, A. Rouault} 
\theoremstyle{plain}

\newtheorem{thm}{Theorem}[section]

\newtheorem{definition}[thm]{Definition}

\newtheorem{lemma}[thm]{Lemma}

\newtheorem{rmk}[thm]{Remark}

\numberwithin{equation}{section}

\def\Id{\mathrm{Id}}
\def\ff{\textbf f}
\def\half{\frac{1}{2}}
\def\bh{\frac{\beta}{2}}
\def\l{\lambda}
\def\s{\sigma}
\def\t{\theta}

\def\b{\beta}
\def\d{\delta}
\def\g{\gamma}
\def\G{\Gamma}
\def\o{\omega}

\def\ex{\mathrm{e}}
\def\N{{\mathbb N}}

\def\Q{{\mathbb Q}}
\def\R{{\mathbb R}}
\def\C{{\mathbb C}}
\def\U{{\mathbb U}}

\def\P{{\mathbb P}}
\def\E{{\mathbb E}}

\begin{document}
\maketitle

\begin{abstract}
The aim of this paper is to give a precise asymptotic description of some eigenvalue statistics stemming from random matrix theory. More precisely, we consider random determinants of the GUE, Laguerre, Uniform Gram and Jacobi beta ensembles and random characteristic polynomials evaluated at $1$ for matrices in the Circular and Circular Jacobi beta ensembles.
We use the framework of mod-Gaussian convergence to provide quantitative estimates of
their logarithmic behavior, as the size of the ensemble grows to infinity. We establish central limit theorems, Berry-Esseen bounds, moderate deviations and local limit theorems. Furthermore, we identify the scale at which the validity of the Gaussian approximation for the tails breaks.\\
With the exception of the Gaussian ensemble, all the results are obtained for a continuous choice of the Dyson parameter, that is for general $\b>0$. The proofs rely on explicit computations which are possible thanks to closed product formulas of the Laplace transforms.
\end{abstract}

60F05, 60F10, 62H10.

\section{Introduction}
\textit{Random determinants}. The distribution of random determinants is a naturally important object in random matrix theory, with many applications in physics and multivariate statistics and, more recently, in numerical analysis, stochastic control theory, finance and telecommunications. As such, its study has a long history.\\
Besides random matrices with independent entries, much attention has been devoted to determinants of symmetric, Hermitian and self-dual matrices, exhibiting a much more challenging dependence structure. The latter groups include the Dyson threefold way $\b=1,2,4$ of classical random matrix models: the Gaussian, the Laguerre and the Jacobi ensembles.
The parameter $\b$ counts the number of real variables needed to specify a single matrix entry (1 for real, 2 for complex and 4 for quaternion numbers).
Dyson observed that, for the above mentioned ensembles, the induced eigenvalue densities correspond to the Boltzmann factor of a log-gas system on the real line at three different inverse temperatures $1, 2$ and $4$. 
For other positive values of $\b$, the log-gas interpretation is still meaningful, but the question regarding the existence of corresponding matrix models was only recently addressed by Dumitriu and Edelman \cite{DE02} and by Killip and Nenciu \cite{KN04}.
Their work brought a renewed interest in the study of the asymptotic behavior of random determinants for general $\b>0$. In the present paper we get some new results in this direction, which we anticipate in the following.\\
Let $A$ be an $n\times n$ matrix with i.i.d. real Gaussian entries and let $A^\dagger$ denote its transpose. The distribution of the real symmetric and positive semidefinite matrix $A^\dagger A$ is called Laguerre (or Wishart) orthogonal ensemble. It describes sample covariance matrices of normally distributed samples. In \cite{J82}, Jonsson discussed the following central limit theorem:
\begin{equation*}
\frac{\log\det \left(A^\dagger A\right)+n+\half\log n}{\sqrt{2\log n}}\rightarrow\mathcal{N}(0,1).
\end{equation*}
This result remains valid when the entries of $A$ are complex or quanternion real Gaussian, replacing the transposed matrix with the Hermitian adjoint, respectively the self-dual matrix. In \cite[Theorem 3.5]{R07}, Rouault showed that actually, the asymptotic log-normality holds true for the $\b$-Laguerre ensemble. The proof is based on the additive structure of the log-determinant. Indeed, by the uniqueness of the Mellin transform, it is possible to establish a Barlett-type theorem which allows to write the determinant as product of independent random variables. \\
An analogous result applies to random matrices distributed according to the $\b$-Uniform Gram or the $\b$-Jacobi ensembles (see \cite[Theorems 3.2 and 3.8]{R07}). \\
The aim of this paper is to get refinements of these central limit theorems. To do so, we provide asymptotic expansions of the Laplace transforms of the log-determinants and we apply the framework of mod-Gaussian convergence developed in~\cite{ FMN16, KN12, JKN11, DKN15, Speed}. This allows to detect the \emph{normality zone}, i.e. the scale up to which the Gaussian approximation for the tails (coming from the CLT) is valid. Moreover, we derive precise (without the $\log$) moderate deviations and local limit theorems. \\
For the $\b$-Gaussian ensemble, we did not succeed in finding asymptotic results valid for general $\b>0$. Indeed many exact formulas which are available for the $\b$-Laguerre, $\b$-Uniform Gram and $\b$-Jacobi ensemble, do not exist for matrices distributed according to the $\b$-Gaussian ensemble. However, for the Gaussian Unitary ensemble ($\b=2$) the Mellin transform of the absolute value of the determinant has been computed explicitly by Mehta and Normand \cite{MN98}. They used sparsity to write this statistics as a product of determinants and discussed its log-normality as the size of the sample growths to infinity. Recently, Tao and Vu \cite{TV12} proposed another proof of such central limit theorem, based on approximating the absolute value of the log-determinant as a sum of weakly dependent terms. We remark that the modulus is needed because, unlike the above mentioned ensembles where the eigenvalues are non-negative, the determinant of GUE depends on the signs of eigenvalues.
Again relying on the mod-Gaussian framework,  we get finer asymptotics and gain new insight into the domain of validity of the Gaussian limit. We thereby recover the Berry-Essen bound and the moderate deviations first established by Eichelsbacher and D\"oring in \cite{DE13} and we additionally provide local limit theorems.\\
It is clear that our studies focus on second order fluctuations; one can refer to \cite{R07}, \cite{AM01} for results of the first order. We shall also remark that, since the logarithm of the determinant  is a linear statistic of the empirical distribution of its eigenvalues, it would be natural to obtain the results mentioned above from those established for the limiting empirical spectral distribution (ESD). This was done for instance in \cite{BS04} at the scale of the central limit theorem.
Our point of view is different and rely on Barlett-type decompositions of the statistics considered, thereby loosing the connection with the ESD. Besides the different approach, our results are consitent with those already present in the literature and obtained by this spectral method (see \cite{R07}, Section 4 for more details).\\
\\
\textit{Random characteristic polynomials}. In circular ensembles, the eigenvalues are located on the unit circle. As a consequence, determinants are complex-valued random variables with unitary norm and the latter quantity, being known explicitly, does not require any further investigation. In this case, a much more informative eigenvalues function to look at is the characteristic polynomial.
One of the sources of interest in such object emerged with the work of Keating and Snaith \cite{KS00}, who used circular unitary matrices to study statistical properties of the zeros of the Riemann zeta function (and other L-functions). Motivated by the results we mention hereafter, we investigate the asymptotic behavior of characteristic polynomials 
for matrices distributed according to the $\b$-circular and the $\b$-circular Jacobi ensemble for all inverse temperatures $\b>0$.\\
Let $U$ be an $n\times n$ unitary matrix distributed according to the Haar measure on the unitary group. Such matrix represents an element of the circular unitary ensemble (CUE). In \cite{KS00}, Keating and Snaith discussed the log-normality of the real and imaginary part of the characteristic polynomial of $U$ evaluated at any point on the unit circle. In particular, they deduced that
\begin{equation*}
\frac{\log\left|\det\left(\Id-U\right)\right|}{\sqrt{\half\log n}}\rightarrow\mathcal{N}(0,1).
\end{equation*}
The proof is based on an exact formula for the Mellin transforms of $\left|\det\left(\Id-U\right)\right|$. The authors also show that the Gaussian distribution appears as limit for real and imaginary parts in the circular orthogonal and symplectic ensembles. See also \cite{BNR07} were analogous results were obtained thanks to a Barlett-type decomposition. 
Besides,  Su \cite{S09} noticed the possibility of extension of
such central limit theorems to all matrices $U$ distributed according to the $\b$-circular ensemble.\\
In this paper, we shall consider the $\b$-circular ensemble as a particular case of the $\b$-circular Jacobi ensemble
introduced in \cite{BNR09}. Indeed, the density of the eigenvalues implied by the latter model contains, as a sub-case, the eigenvalues density of the $\b$-circular ensemble.
In \cite{BNR09}, the authors were able to write the characteristic polynomial evaluated at one of $\b$-circular Jacobi matrices as a product of independent complex
random variables. This functional was then further investigated in its first and second order asymptotical behavior in \cite{NNR13}. 
Relying on additive structure of its logarithm and using the mod-Gaussian setting, we establish asymptotic results (central limit theorem, moderate deviations, local limit theorems) for the real part of such log-characteristic polynomial evaluated at $1$.
As a particular case, we obtain the same limiting theorems for the $\b$-circular ensemble, thereby recovering the above mentioned central limit theorem. We should mention that,  for Dyson's circular ensembles, mod-Gaussian convergence has been first established in \cite{KN12}.\\
\\
The structure of the paper goes as follows. In the coming section we recall the definition of mod-Gaussian convergence and the limiting theorems this convergence implies. We then present in section \ref{section:matrixmodels} the matrix models (Laguerre, Uniform Gram, Jacobi, Circular and Circular Jacobi) for the classical choices $\b=1,2,4$ of the Dyson parameter and their extensions to general $\b>0$. In addition, we recall the known facts about determinants and characteristic polynomials present in the literature. The main results of this article are in section \ref{section:results} and their proofs can be found in section \ref{section:proofs}. All such proofs rely on some fundamental theorems which are established in section \ref{section:asymptotics}.\\
\\
\textbf{Notation.} By $f=O(g)$ for $x\in X$, where $X$ is an arbitrary set on which $f$ is defined, we mean that there exists a constant $K$ (which may be a function of other parameters) such that $|f(x)|\le K|g(x)|$ for all $x\in X$. Any constant $K$ for which this hold is called \textit{implied constant}. The implied constant is said to be \textit{absolute} if it does not depend on further quantities. If a set $X$ is not explicitly given then the estimate is
assumed to hold for all sufficiently large values of the variable involved.

\section{Mod-Gaussian convergence and limiting results}
\label{section:modGaussian}
The framework of mod-Gaussian convergence has been introduced and developed by Delbaen, F\'eray, Jacod, Kowalski, M\'eliot and Nikeghbali  in \cite{FMN16, KN12, JKN11, DKN15, Speed}. Consider a sequence of random variables $\left(X_n\right)_{n\in\N}$, whose sequence of characteristic functions does not converge. It is well known that, by L\'evy continuity theorem, this prevents the $X_n$'s to converge in distribution. The main idea underpinning mod-Gaussian convergence, is to compare the characteristic functions of the $X_n$'s with the one of a Gaussian random variable with variance growing with $n$ and to see if the ratio of such Fourier transforms converges to a non-trivial limiting object $\psi$.
Intuitively, we try to see wether $X_n$ can be thought of as a Gaussian variable with growing variance, plus a residual noise which is encoded in the limiting function. Note that, in this setting, $\psi$ is not necessarily the characteristic function of a random variable.\\
Although mod-Gaussian convergence entails
much more information, for the aim of this paper, it will be an instrument to deduce asymptotic results, like a central limit theorem, moderate deviations and local limit theorems.
In what follows we summarize some of the limiting results that mod-Gaussian convergence implies.\\
\\
For $-\infty \leq c < 0 < d\leq \infty$, set
\begin{equation*}
S_{(c,d)}=\{z\in\C, c < \Re (z) < d\}\ , \qquad S_d = S_{-d , \infty}.
\end{equation*}

\begin{definition}\label{def:modGaussian}
Let $\left(X_n\right)_{n\in\N}$ be a sequence of real valued random variables, with moment generating functions all existing over the strip $S_{c,d}$. 
One says that $\left(X_n\right)_{n\in\N}$ converges mod-Gaussian on $S_{(c,d)}$, with parameters $t_n$ and limiting function $\psi:S_{c,d}\rightarrow\C$ if, locally uniformly on $S_{(c,d)}$,
\begin{equation*}
\psi_n(z):=\E\left[\ex^{zX_n}\right]\ex^{-t_n\frac{z^2}{2}}\longrightarrow\psi(z),
\end{equation*}
where $\left(t_n\right)_{n\in\N}$ is some deterministic sequence going to $+\infty$ with $n$.
\end{definition}

Throughout the remaining part of this section, $\left(X_n\right)_{n\in\N}$ denotes a sequence converging mod-Gaussian on $S_{(c,d)},$ with parameters $t_n$ and limiting function $\psi$. It is straightforward to see that mod-Gaussian convergence implies a central limit theorem for a proper renormalization of the sequence $\left(X_n\right)_{n\in\N}$. Indeed, for all $\xi\in\R$, we have
\begin{equation*}
\E\left[\ex^{i\xi\frac{X_n}{\sqrt{t_n}}}\right]=\ex^{\frac{\xi^2}{2}}\psi_n\left(\frac{i\xi}{\sqrt{t_n}}\right)=\ex^{\frac{\xi^2}{2}}\psi(0)(1+o(1))=\ex^{\frac{\xi^2}{2}}(1+o(1))
\end{equation*}
thanks to the uniform convergence of $\psi_n$ towards $\psi$. This implies, together with the L\'evy's continuity theorem, that
\begin{equation*}
Y_n:=\frac{X_n}{\sqrt{t_n}}\rightarrow\mathcal{N}(0,1)
\end{equation*}
or equivalently that 
\begin{equation}\label{cltratio}
\lim_{n\to\infty}\frac{\P\left[Y_n\ge x\right]}{\P\left[\mathcal{N}(0,1)\ge x\right]}=1.
\end{equation}
Actually, many other quantitative estimates which completely describe the flu ctuations of the normalized sequence $\left(Y_n\right)_{n\in\N}$, are consequences of the notion of mod-convergence. For instance, one can show that the \emph{normality zone} is of order $o(t_n)$, this meaning that the limit (\ref{cltratio}) holds true for any $x=o(t_n)$.
At the edges of such zone, i.e. at scale $O(t_n)$, this approximation breaks and the residue $\psi$ describes how to correct the Gaussian approximation of the tails.
These statements are made more precise in the following two theorems. We refer to \cite{FMN16} for a detailed proof of such results.

\begin{thm}[Central limit theorem up to scale $o(t_n)$, Theorem 4.3.1 in \cite{FMN16}]\ 
\label{extendedclt}\\
For  $x=o\left(\sqrt{t_n}\right)$,
\begin{equation*}
\P\left[\frac{X_n}{\sqrt{t_n}}\ge x\right]=\P\left[
\mathcal{N}(0,1)\ge x\right]\left(1+o(1)\right)=\frac{\ex^{-\frac{x^2}{2}}}{x\sqrt{2\pi}}\left(1+o(1)\right).
\end{equation*}
\end{thm}

\begin{thm}[Precise moderate deviations in the scale $O(t_n)$ Theorem 4.2.1 in \cite{FMN16}]\ 
\label{preciseldp}\\
For $x\in (0,d)$,
\begin{equation*}
\P\left[X_n\ge t_nx\right]=\frac{\ex^{-t_n\frac{x^2}{2}}}{x\sqrt{2\pi t_n}}\psi(x)\left(1+o(1)\right),
\end{equation*}
and for $x\in(c,0)$,
\begin{equation*}
\P\left[X_n\le t_nx\right]=\frac{\ex^{-t_n\frac{x^2}{2}}}{|x|\sqrt{2\pi t_n}}\psi(x)\left(1+o(1)\right).
\end{equation*}
\end{thm}

Under the additional assumption of mod-Gaussian convergence with a \emph{zone of control}, it is possible to deduce the speed of convergence of the above mentioned CLT. Such notion has been introduced in \cite[Definition 4, Remark 2]{Speed} in the context of mod-stable convergence.

\begin{definition} 
Consider the following statements.
\begin{enumerate}[label=(\textbf{Z\arabic*}),ref=(Z\arabic*)]
\item \label{Z1} Fix $v \ge 1$, $w>0$ and $\g\in\R$. There exists a zone of convergence $[-Dt_n^\g,Dt_n^\g]$, $D>0$, such that for all $\xi$ in this zone, $\left|\psi_n(i\xi)-1\right|$ is smaller than
\begin{equation}
K_1|\xi|^v e^{K_2 |\xi|^w} \nonumber
\end{equation}
for some positive constants $K_1$ and $K_2$, that are independent of $n$.
\item \label{Z2} One has 
\begin{equation}
w \geq 2;\qquad -\half<\g\le\frac{1}{w-2};\qquad D\le \left(\frac{1}{4K_2}\right)^{\frac{1}{w-2}}. \nonumber
\end{equation}
\end{enumerate}
If Conditions (Z1) holds for some parameters $\g>-\half$ and $\nu, w, D, K_1, K_2$, then 
(Z2) can always be forced by increasing $w$, and then decreasing $D$ and $\g$. If Conditions (Z1) and (Z2) are satisfied,  we say that we have mod-Gaussian convergence for the sequence $\left(X_n\right)_{n\in\N}$ with zone of control $[-Dt_n^\g,Dt_n^\g]$ and index of control $(v,w)$. 
\end{definition}

\begin{thm}[Speed of convergence, Theorem 15 in \cite{Speed}]\ \label{speed}\\
Assume that conditions (Z1) and (Z2) hold, together with the inequality
 $\g\le \frac{v-1}{2}$. Then,
\begin{equation*}
d_{Kol}\left(\frac{X_n}{\sqrt{t_n}}, \mathcal{N}(0,1)\right)\le C(D, v,K_1)\frac{1}{t_n^{\g+\half}},
\end{equation*}
where $d_{Kol}(\cdot, \cdot)$ is the Kolmogorov distance and 
\begin{equation}
\label{defc}
C(D,v,K_1) = \frac{3}{2 \pi} \left(2^{v-1} \G\left(\frac{v}{2} \right)K_1 + \frac{7}{D} \sqrt{\frac{\pi}{2}}\right).
\end{equation}
\end{thm}

In the other direction, we can also look at the fine scale behavior of $X_n$ by trying to control probabilities such as
\begin{equation*}
\P\left[X_n-x_n\in (a,b)\right],
\end{equation*}
where $x_n$ can grow with $n$ and $a$ and $b$ are fixed real numbers. This result has been established in \cite{}, relying on the notion of zone of control.

\begin{thm}[Local limit theorem, Theorem  in \cite{FMN17}]\ \label{llt}\\
 Let $x\in\R$ and $(a,b)$ be a fixed interval, with $a<b$. 
 Assume that conditions (Z1) and (Z2) hold. Then for every exponent $\d\in\left(0,\g+\frac{1}{2}\right)$, 
 \begin{equation*}
\lim_{n\to\infty}(t_n)^{\d}\P\left[Y_n-x\in\frac{1}{t_n^\d} (a,b)\right]=\frac{b-a}{\sqrt{2\pi}}.
\end{equation*}
In particular, assuming $\g>0$, with $\d=\half$ one obtains
 \begin{equation*}
\lim_{n\to\infty}(t_n)^{\half}\P\left[X_n-x\left(t_n\right)^\half \in (a,b)\right]=\frac{b-a}{\sqrt{2\pi}}.
\end{equation*}
\end{thm}
We shall see that all the eigenvalues statistics considered satisfy the definition (\ref{def:modGaussian}), together with conditions \ref{Z1} and \ref{Z2}.

\section{The presentation of the  eigenvalues statistics}
\label{section:matrixmodels}
In this long section we present our different matrix models. Their common feature is to introduce random determinants or random characteristic polynomials which can be written as a product of independent random variables. Such product-structure comes from Barlett-type theorems for real eigenvalues matrix models and from the Verblunsky and modified Verblunsky coefficients in the circular case. The eigenvalues statistics we consider are the logarithms of such random determinants and characteristic polynomials. In particular, we shall recall the closed formulas for their Laplace transforms (or equivalently for the Mellin transforms of determinants and characteristic polynomials themselves), which are the key ingredient in order to show mod-Gaussian convergence.
\subsection{Random determinants}
In the random matrix context, the choice of the $\mathbb{R}$-division algebra on which the entries are defined, corresponds to a different parameter $\b$ in the joint eigenvalue density, namely $\b=1$ for real entries, $\b=2$ for complex entries and $\b=4$ for the quaternions.
For such classical values of the Dyson's parameter, many of the models we are interested in, can be constructed starting from a Gaussian random matrix $A$. This is a matrix with independent and identically distributed entries, which are real, complex or quaternion standard normal.
Throughout, we denote by $A^\dagger$ the transpose, the hermitian conjugate or the dual of $A$ according as $A$ is real, complex or quaternion.\\
\\
\textit{Hermite or Gaussian ensemble}. Let $A$ be an $n \times n$ Gaussian matrix over $\R, \C$ or $\mathbb{H}$.
The distribution of the $n \times n$ random matrix $\frac{A+A^\dagger}{2}$ is called  the Gaussian orthogonal (GOE), unitary (GUE), respectively  symplectic  ensemble (GSE).  The joint density function of the eigenvalues is given by
\begin{equation}\label{Her-density}
\ff_n^{H, \b}(\l_1, \dots, \l_n) = \frac{1}{Z_n^H(\b) } \prod_{1 \leq j < k\leq n} |\l_k - \l_j|^\b \prod_{k=1}^n \ex^{-\frac{\l_k^2}{2}}
\end{equation}
for $\b=1,2$ and $4$, with normalizing constant \begin{equation*}
Z_n^H(\b) = (2\pi)^{n/2}  \prod_{k=1}^n \frac{\G\left(1+k\bh\right)}{\G\left(1+\bh\right)}.
\end{equation*} 
Using a tridiagonal reduction algorithm, Dumitriu and Edelman \cite{DE02} proved that any other choice of $\b>0$ in (\ref{Her-density}) corresponds to a matrix model constructed with entries from classical distributions. However, the lack of appropriate Selberg integrals, makes it much harder to investigate the distribution of determinants in all $\b$-Hermite ensembles.
Our study is restricted to GUE random matrices, which we shall denote by $W_n^H$. For such model, Mehta and Normand \cite{MN98} computed explicitly the Mellin transform of the absolute value of the determinant. They obtained the following formula
\begin{equation}\label{Her-Mellin}
\E\left[\left|\det W^H_n\right|^z\right]=2^{\frac{nz}{2}}\prod_{k=1}^{n}\frac{\G\left(\frac{z+1}{2}+\lfloor \frac{k}{2}\rfloor\right)}{\G\left(\frac{1}{2}+\lfloor \frac{k}{2}\rfloor\right)}
\end{equation}
which is well defined for any $z\in\C$ with $\Re (z)>-1.$\\
Although they did not interpret the factors in (\ref{Her-Mellin}) probabilistically, they are the Mellin transforms of $\chi$-distributed random variables. By uniqueness of the Mellin trasform, this implies that
\begin{equation*}
\E\left[\left|\det W^H_n\right|^z\right]\stackrel{(d)}{=}\prod_{k=1}^n \rho_{k,n}^{H}
\end{equation*}
where for $k=1,\ldots,n$, $\rho_{k,n}^{H}$ are independent and
\begin{equation*}
\rho_{k,n}^{H} \sim \chi_2\left(\lfloor \frac{k}{2}\rfloor+1\right).
\end{equation*}
This decomposition was obtained by Edelman and La Croix (\cite{Edel1}) via  describing $\left|\det W^H_n\right|$ as the product of the GUE singular values.\\
\\
\textit{Laguerre or Wishart ensemble.}
Let $A$ be an $n \times r$, with $r \le n$,  Gaussian matrix over $\R, \C$ or $\mathbb{H}$.
The distribution of the $r \times r$ random matrix $A^\dagger A$ is called Laguerre real, complex, symplectic ensemble. Clearly, the product $A^\dagger A$ is positive semi-definite, hence its eigenvalues $\left(\l_1, \dots, \l_r\right)$ are real and non negative. Their joint density function on the set $(0,\infty)^r$ is proportional to 
\begin{equation*}
\prod_{1 \leq j < k \leq r} |\l_k - \l_j|^{\b}\prod_{k=1}^r \left( \l_k^{\bh(n-r+1)-1} \ex^{-\frac{\l_k}{2}}\right)
\end{equation*}
for $\b=1,2$ and $4$ respectively. 
Tridiagonal models for the $\b$-Laguerre ensembles have been constructed in \cite{DE02}. In the present article, we discuss the regime $r=n$ and denote by $W_n^{L,\b}$ the $\b$-Laguerre distributed random matrix of dimension $n\times n$. In this case, the joint eigenvalues density becomes
\begin{equation*}
\ff^{L, \b}_{n}\left(\l_1, \dots, \l_n\right) = \frac{1}{Z^{L}_{n} \left( \b \right)} \prod_{1 \leq j < k \leq n} |\l_k - \l_j|^{\b}\prod_{k=1}^n \left( \l_k^{\bh-1} \ex^{-\frac{ \l_k}{2}}\right) ,
\end{equation*}
where the normalization constant $Z^{L}_{n} \left( \b \right)$ can be calculated using the Selberg integral (17.6.5) in \cite{Mehta} and equals
\begin{equation*}
Z_n^L(\b) =2^{\bh n^2}\prod_{k=1}^n\frac{\G\left(1+\bh k\right)\G\left(\bh k\right)}{\G\left(1+\bh\right)}.
\end{equation*}
Using the same Selberg formula, it is straightforward to deduce that,
\begin{equation}\label{Lag-Mellin}
\E\left[\left(\det W_n^{L, \b}\right)^z\right]=2^{nz} \prod_{k=1}^n \frac{\G\left(\bh k+z\right) }{\G\left(\bh k\right)} 
\end{equation}
for any $\b>0$.
This implies the following Barlett-type decomposition
\begin{equation*}
\det W^{L,\b}_n\stackrel{(d)}{=}\prod_{k=1}^r \rho_{k,n}^{L,\b},
\end{equation*} 
where for $k=1,\ldots,r$ the variables $\rho_{k,n}^{L,\b}$  are independent and
\begin{equation*}
\rho_{k,n}^{L,\b}\sim \text{Gamma}\left(\bh(n-j+1),\half\right).
\end{equation*} 

\textit{Jacobi or MANOVA ensemble.}
Let $A_1$ and $A_2$  be $n_1\times r$ and $n_2 \times r$  Gaussian matrices over $\R, \C$ or $\mathbb{H}$ with $r \leq \min(n_1, n_2)$.  The distribution of the $r \times r$ matrix
\begin{equation}\label{Jac-matrix}
\left(M_1^\dagger M_1 + M_2^\dagger M_2\right) ^{-1} M_1^\dagger M_1
\end{equation}
is called Jacobi real, complex, respectively symplectic ensemble. As for the previous matrix models, it can be generalized to all $\b >0$ (see \cite{DE02}, \cite{KN04} for the corresponding tridiagonal matrix models). The joint density of the eigenvalues on the set $(0,1)^r$ is given by
\begin{equation*}
f_{r, n_1, n_2}^{J,\b}(\l_1,\ldots,\l_r)=\frac{1}{Z_{r,n_1,n_2}^{J} (\b)}\prod_{1 \le k < j \le r} |\l_j - \l_k|^{\b}\prod_{k=1}^{r} \l_k^{\bh(n_1-r+1)-1} \left( 1 - \l_k \right)^{\bh (n_2-r+1)-1}
\end{equation*}
where the normalization $Z_{r,n_1,n_2}^{J} (\b)$ can be obtained by means of the Selberg integral (17.1.3)  in \cite{Mehta},
\begin{equation*}
Z_{r,n_1,n_2}^{J} (\b)= \prod_{k=1}^{r} \frac{\G\left( 1 + \bh k\right) \G\left( \bh(n_2-r+k)\right) \G\left(\bh(n_1-r+k )\right)}{\G\left( 1+\bh\right) \G\left( \bh \left( n_1+n_2 + k - r\right)\right)}.
\end{equation*}
It is possible to extend the study of the determinant to   $n_2\le r \le n_1$, when the matrix $A_2^\dagger A_2$ is singular. In this case, (\ref{Jac-matrix}) has $1$ as eigenvalue of multiplicity $r-n_2$ and its distribution has therefore no density. However, the density of the non-one eigenvalues is given by $f_{n_2, n_1+n_2-r, r}^{J,\b}$. \\
Motivated by this observation, we consider the following extension of the above model. For every $\b>0$, define the family of joint probability density functions on $(0,1)^{\min (n_2,r)}$ as follows
\begin{align*}
\ff^{J, \b}_{r, n_1, n_2}(\l_1,\ldots,\l_r) := \begin{cases} f_{r, n_1, n_2}^{J,\b}(\l_1,\ldots,\l_r), &\mathrm{if} \quad r\leq \min(n_1, n_2), \\ f^{J,\b}_{n_2, n_1+n_2-r,r}(\l_1,\ldots,\l_r) &\mathrm{if} \quad n_2 \le r \le n_1.\end{cases}
\end{align*}
We fix $\tau_1, \tau_2 >0$ and discuss the regime $r=n_1=\lfloor n \tau_1 \rfloor$ and $n_2 = \lfloor n \tau_2\rfloor$. We set by convention
\begin{equation*}
\det W^{J,\b}_n=\prod_{k=1}^{\min(\lfloor n \tau_2\rfloor,\lfloor n \tau_1\rfloor)} \l_k
\end{equation*}
where the $\l_k$'s are distributed according to $\ff^{J, \b}_{\lfloor n \tau_1\rfloor, \lfloor n \tau_1\rfloor,\lfloor n \tau_2\rfloor}$ and we call it determinant even if there is no matrix behind. Its Mellin transform is well defined for any $z\in\C$, with $\Re (z)>-\bh$ and given by
\begin{equation}\label{Jac-Mellin}
\E \left[\left( \det W^{J,\b}_n\right)  ^z\right]  =\prod_{k=1}^{\lfloor n \tau_1 \rfloor} \frac{\G \left( \bh (\lfloor n \tau_2 \rfloor + k) \right) \G\left( \bh k + z\right)}{\G \left(\bh k \right) \G \left( \bh (\lfloor n \tau_2 \rfloor + k) + z \right)}.
\end{equation}

\textit{Uniform Gram ensemble.}  Let $B$ be an $n\times r$, with $r\le n$, matrix with independent columns, uniformly distributed over the real,  complex or quaternion unit sphere. The distribution of the $r\times r$  matrix $B^\dagger B$ is called the real, complex, respectively symplectic Uniform Gram ensemble. In particular, since  the diagonal entries of $B^\dagger B$ are one (see \cite{M93, M97, GN00, R07}), it describes sample correlation matrices. Even if no explicit formula for the joint density of eigenvalues is available, the expression 
\begin{equation}\label{Gram-density}
\frac{1}{Z_{r,n}^{G} (\b)} (\det B^\dagger B)^{\bh(n-r+1)-1},
\end{equation}
where
\begin{equation*}
Z_{r,n}^{G} (\b)= \pi^{\bh r (r-1)} \prod_{k=1}^r \frac{\G\left( \bh (n-k+1)\right)}{\G\left( \bh n \right)},
\end{equation*}
is a density on the space of symmetric ($\b=1$), Hermitian ($\b=2$) matrices and self-dual ($\b=4$) which fits with the distribution of correlation matrices in the real, complex and quaternion case. For $r=n$, denote by $W_n^{G,\b}$ a random matrix of dimension $n\times n$ distributed according to the real, complex and symplectic Uniform ensemble. The density (\ref{Gram-density}) yields  the Mellin transform,
\begin{equation}\label{Gram-Mellin}
\E\left[\left(\det W^{G, \b}_n\right)^z\right]=\prod_{k=1}^n \frac{\G\left(\bh k +z\right) \G\left(\bh n\right)}{\G\left(\bh k\right)\G\left(\bh n+z\right)}
\end{equation}
(see \cite{GN00} Exercise 3.26 p.130 for the real case). As a consequence for $\b=1,2,4$, the following Barlett-type decomposition holds
\begin{equation}\label{Gram-Barlett}
\det W_n^{G,\b}\stackrel{(d)}{=}\prod_{k=2}^r \rho_{k,n}^{G,\b}
\end{equation}
where the variables $\rho_{k,n}^{G,\b}$ for $k=2,\ldots,r$ are independent and 
\begin{equation*}
\rho_{k,n}^{G,\b}\sim \text{Beta}\left(\bh(n-k+1),\bh(k-1)\right).
\end{equation*}
Since product in (\ref{Gram-Barlett}) is meaningful for every $\b>0$, we define $\det W^{G, \b}_n$ as a random variable distributed according to the right hand side of (\ref{Gram-Barlett}) also for $\b\not = 1,2,4$.

\subsection{Characteristic polynomials} 

\textit{Circular ensemble.} Dyson's circular ensembles are measures on spaces of unitary matrices. The circular unitary ensemble (CUE) is the the unitary group $\U(n)$ endowed with its Haar measure. The circular orthogonal ensemble (COE) is the subset of $\U(n)$ consisting of symmetric matrices together with the unique measure invariant under the map $U\rightarrow W^\dagger UW$, with $W\in\U(n)$. The symplectic one (CSE) is the subset of $\U(2n)$ consisting of self-dual matrices equipped with the measure invariant under the map $U\rightarrow W^RUW$, with $W\in\U(2n)$ and where
\begin{equation*}
W^R:=
\begin{bmatrix} 
0 & 1 & & & \\
-1 &  0 & & &\\
 & & \ldots & &\\
 & & & & 0 & 1\\
 & & & & -1  & 0
\end{bmatrix}^T  W^T
\begin{bmatrix} 
0 & 1 & & & \\
-1 &  0 & & &\\
 & & \ldots & &\\
 & & & & 0 & 1\\
 & & & & -1  & 0.
\end{bmatrix}
\end{equation*}
The eigenvalues  $\left(e^{i \t_1},\ldots,e^{i \t_n}\right)$ are located on the unit circle and the induced probability measure on the phases $\left(\t_1,\ldots,\t_n\right)$ is given by 
\begin{align}\label{Circ-density}
\ff_n^{C, \b}\left( \t_1, \dots, \t_n \right)=\frac{1}{Z^{C}_n(\b)} \prod_{1 \leq k < j \leq n} \left| \ex^{i \t_k} - \ex^{i \t_j}\right|^{\b},
\end{align} 
for $\b=1,2$ and $4$ respectively, with normalization constant 
\begin{align*}
Z_{n}^{C} (\b)= (2\pi)^{n}\frac{\G\left( \bh n+1\right)}{\G\left( \bh +1 \right)^n}.
\end{align*}
For general $\b>0,$ Killip and Nenciu  \cite{KN04} constructed random matrices with joint eigenvalues density given by (\ref{Circ-density}). For a given unitary matrix, they considered the associated spectral measure $\mu$, together with its Verblunsky coefficients. Applying the Gram-Schmidt procedure to the set of polynomials $\{1, z, z^2,\ldots\}$ in $\mathrm{L}^2(\mathbb{S}^1, \mu)$, they got an orthonormal basis and showed that, if one takes independent Verblunsky coefficients with a specific distribution, the matrix representing the operator $f(z)\rightarrow zf(z)$ in such basis is matrix model for the $\b$-circular ensemble. Actually Killip and Nenciu provided a matrix model which is much sparser (five-diagonal) than the latter one. This is obtained by the representing the above operator in another orthonormal basis, formed from the set $\{1,z,z^{-1},z^2,z^{-2},\ldots\}$.\\
In the following, if $W_n^{C,\b}$ is a matrix distributed according to the $\b$-Circular ensemble, we set
$Z_n^{\b}:=\det(\Id_n-W_n^{C,\b})$. A closed formula for Mellin transform of the real part of $\log Z_n^{C,\b}$ has been computed in \cite{S09}:
\begin{align}\label{Circ-Mellin}
\E\left[|Z_n^{\b}|^z\right]&=\prod_{k=0}^{n-1}\frac{\G\left(\bh k+1\right)\G\left(\bh k+1+z\right)}{\G\left(\bh k+1+\frac{z}{2}\right)^2}. 
\end{align} 
However, as mentioned in the introduction, we can recover this expression from the results obtained for the $\b$-Circular Jacobi ensemble, which we introduce hereafter.
\\
\\
\textit{Circular Jacobi ensemble.} The $\b$-circular Jacobi ensemble is the distribution on the unitary group $\U(n)$ such that the arguments of the eigenvalues $\left(\ex^{i \t_1},\ldots,\ex^{i \t_n}\right)$  have density
\begin{equation}\label{JacCirc-density}
\ff_n^{CJ, \b, \d}\left( \t_1, \dots, \t_n \right)=\frac{1}{Z^{CJ}_n(\b)} \prod_{1 \leq k < j \leq n} \left| \ex^{i \t_k} - \ex^{i \t_j}\right|^{\b} \prod_{k=1}^n \left( 1 - \ex^{-i \t_k}\right)^\d \left( 1 - \ex^{i \t_k}\right)^{\overline{\d}}
\end{equation} 
for $\b>0$ and $\d \in \C$, $\Re(\d) > -\frac{1}{3}.$ Note that, for $\d=0$, (\ref{JacCirc-density}) coincide with the $\b$-circular ensemble. Such model has been introduced in \cite{BNR09}, where the authors provide the corresponding matrix models. To do so, they use an argument similar to the one of Killip and Nenciu. However, since for this model the Verblunsky coefficients are no more independent, they introduce a new set of parameters $\left(\g_k\right)_k$, called deformed Verblunsky coefficients, which mantains such independence property.\\
They also show that, if $W_n^{CJ,\b,\d}$ denotes a unitary matrix in the $\b$-Circular Jacobi ensemble, then
\begin{equation*}
Z_n^{^{\b,\d}}:=\det (\Id_n-W_n^{CJ,\b,\d})=\prod_{k=0}^{n-1}\left(1-\g_k\right)
\end{equation*}
where for $k=0,\ldots, n-1$ the $\g_k$'s are independent. Moreover, they found an explicit probability distribution for the modified Verblunsky coefficients, from which one can recover the characteristic polynomial of the $\b$-Circular Jacobi Ensemble. In particular, for $k=0,\ldots, n-2$ the density of the $\g_k$'s with respect to the Lebesgue measure $d^2z$ on $\C$ is proportional to
\begin{equation*}
\left(1-|z|^2\right)^{\bh\left(n-k-1\right)-1}(1-z)^{\bar{\d}}(1-\bar{z})^\d \mathrm{1}_{|z|<1}(z)
\end{equation*}
and the density  of $\gamma_{n-1}$ with respect to the Haar measure on $\mathbb{S}^1$ is proportional to
\begin{equation*}
(1-z)^{\bar{\d}}(1-\bar{z})^\d \mathrm{1}_{|z|=1}(z).
\end{equation*}
As a consequence, they obtain (see \cite{BNR09}, Formula 4.2) the following Mellin transform 
\begin{align}\label{JacCirc-Mellin}
\E\left[|Z_n^{^{\b,\d}}|^z\right]&=\prod_{k=0}^{n-1}\frac{\G\left(\bh k+1+\d\right)}{\G\left(\bh k+1+\d+\frac{z}{2}\right)}\frac{\G\left(\bh k+1+\bar{\d}\right)}{\G\left(\bh k+1+\bar{\d}+\frac{z}{2}\right)}\nonumber \\
&\quad \cdot \prod_{k=0}^{n-1}  \frac{\G\left(\bh k+1+\d+\bar{\d}+z\right)}{\G\left(\bh k+1+\d+\bar{\d}\right)}, 
\end{align}
which is analytic for $z\in\C$ with $\Re (z)> - \frac{1}{3}.$

\section{Limiting results}
\label{section:results}

We first focus on the study of random determinants and then present the results for random characteristic polynomials.
In the statement of the theorems below, $G$ denotes the Barnes $G$-function (see also Appendix \ref{appendix}) and $\Upsilon$ is the function defined in (\ref{upsilon}), which we recall here for completeness:
\begin{align*}
\Upsilon(z)&:=\bh \log G\left(\frac{2z}{\b}+1\right)-\left(z-\half\right)\log\G\left(\frac{2z}{\b}+1\right)\\
&\quad +\int_0^\infty\left( \frac{1}{2s} - \frac{1}{s^2} + \frac{1}{s\left(\ex^s-1\right)}\right)\frac{\ex^{-sz}-1}{\ex^{s\bh}-1}ds+\frac{z^2}{\b}+\frac{z}{2}. 
\end{align*}

\subsection{Random determinants}\label{subs:determinantresults}
The coming limiting theorems are obtained as consequences of mod-Gaussian convergence of the sequences $\log|\det W_n^H|$ and $\log(\det W_n^{i,\b})$ for $i=L,J,G$. For this reason, we first prove that all such sequences verify the Definition \ref{def:modGaussian}. To do so,  we investigate the asymptotic  behavior of their Laplace transforms. These expansions are obtained in the next two lemmas.
\begin{lemma}\label{GUE-lemma}
The cumulant generating function of the modulus of the GUE log-determinant satisfies
\begin{equation*}
\log \E\left[e^{z\log \left|\det W_{n}^{H} \right|}\right] = z \mu^{H} _{ n}  + \frac{z^2}{4} \log\left(\frac{n}{2}\right)+ \Upsilon^H(z)+o(1),
\end{equation*}
locally uniformly on the band $S_1$, where 
\begin{equation*}
\Upsilon^H(z):=\log \left(\frac{\G\left(\frac{1}{2}\right)}{\G\left(\frac{z+1}{2}\right)}\frac{G\left(\frac{1}{2}\right)^2}{G\left(\frac{z+1}{2}\right)^2} \right),
\end{equation*}
and
\begin{equation*}
 \mu^{H} _{n} := \half\log 2\pi-\frac{n}{2}+\frac{n}{2}\log n.
\end{equation*}
\end{lemma}

\begin{lemma}\label{Determinant-lemma}
The cumulant generating functions of the $\b$-Laguerre, the $\b$-Uniform Gram and the $\b$-Jacobi log-determinant satisfy
\begin{align*}
\log \E\left[e^{z\log \left(\det W_{n}^{L,\b}\right)}\right] &= z \mu^{L} _{\b, n}  + \frac{z^2}{\b} \log n+ \Upsilon(z)+o(1), \\
\log \E\left[e^{z\log \left( \det  W_{n}^{G,\b} \right)}\right]  &= z \mu^{G} _{\b, n} + \frac{z^2}{\b} \log n-\frac{z^2}{\b}+ \Upsilon(z)+o(1),\\
\log \E\left[e^{z\log \left(\det W_n^{J,\b}\right)}\right]  &= z \mu^{J} _{\b, n} + \frac{z^2}{\b} \log n + \frac{z^2}{\b}\log\left(\frac{\tau_1\tau_2}{\tau_1+\tau_2}\right) + \Upsilon(z)+o(1)
\end{align*}
locally uniformly on the band $S_{\bh}$, where 
\begin{align*}
 \mu^{L} _{\b, n} :&= \left(\half - \frac{1}{\b}\right) \log n - n + n \log (\b n), \\
\mu^{G} _{\b, n}  :&=\left(\half - \frac{1}{\b}\right) \log n - n +\frac{1}{\b},\\
\mu^{J} _{\b, n} :&=\left(\half - \frac{1}{\b}\right) \log n + \left(\half - \frac{1}{\b}\right)\log\left(\frac{\tau_1\tau_2}{\tau_1+\tau_2}\right)-\mathcal{\varepsilon }(\lfloor n \tau_1\rfloor ,\lfloor n \tau_2\rfloor )-\mathcal{\varepsilon }(\lfloor n \tau_2\rfloor ,\lfloor n \tau_1\rfloor )
\end{align*}
with $\mathcal{\varepsilon}(x,y):=x\log \left( 1+\frac{y}{x}\right).$
\end{lemma}

\begin{rmk} Note that the expectations of the log-determinants of the above ensembles agree with the corresponding $\mu$ up to some constant.
\end{rmk}

The lemmas \ref{GUE-lemma} and \ref{Determinant-lemma} directly imply mod-Gaussian convergence for the shifted log-determinants of the Gaussian unitary ensemble and of the considered $\b$-ensembles.

\begin{thm}[Mod-Gaussian convergence for the log-determinant of GUE]\ \\
Let 
\begin{equation*}
X_n^{H} := \log \left|\det W_n^{H} \right| - \mu_{n}^H.
\end{equation*}
Then as $n \to \infty$, the sequence $\left( X_n^{H}\right)_{n \in \N}$ converges mod-Gaussian on the strip $S_{1}$  with parameters $t_n^{H} = \half \log\left(\frac{n}{2}\right)$ and limiting function 
\begin{equation*}
\psi^{H}(z) =\frac{\G\left(\frac{1}{2}\right)}{\G\left(\frac{z+1}{2}\right)}\frac{G\left(\frac{1}{2}\right)^2}{G\left(\frac{z+1}{2}\right)^2} .
\end{equation*}
\end{thm}

\begin{thm}[Mod-Gaussian convergence for the log-determinants of $\b$-ensembles]\ \\
For $i=L,J,G$, let 
\begin{equation*}
X_n^{i, \b} := \log\left( \det W_n^{i, \b}\right) - \mu_{\b,n}^i.
\end{equation*}
Then as $n \to \infty$, the sequences $\left( X_n^{i, \b}\right)_{n \in \N},\ i=L,J,G$ converge mod-Gaussian on the strip $S_{\bh}$  with parameters $t_n^{\b} =\frac{2}{\b}\log n$
and the limiting functions 
\begin{align*}
\psi^{L,\b}(z) =&\exp\left(\Upsilon(z)\right), \\
\psi^{G,\b}(z) =&\exp\left(\Upsilon(z) - \frac{z^2}{\b}\right),\\
\psi^{J,\b}(z) =& \exp\left(\Upsilon(z) + \frac{z^2}{\b'}\log\left(\frac{\tau_1\tau_2}{\tau_1+\tau_2}\right)\right). \\
\end{align*}
\end{thm}
Next we state the 
 deviation results and the extended central limit theorems that follow directly from mod-Gaussian convergence. We recall that Cram\'er type moderate deviations, Berry-Esseen bounds and moderate deviation principles for the log-determinant of GUE were first derived in \cite{DE13} using bounds on cumulants. Up to constant terms in the shift and in the normalization, their sequence $\left(W_{n,2}\right)_{n\in\N}$ corresponds to our sequence $\left(X_n^H/\sqrt{t_n}\right)_{n\in\N}$. 
The precise 
moderate deviation principle of Theorem \ref{GUE-pldp} can be compared with their Cram\'er type moderate deviations results for $x=C_1\s_\b$. Furthermore, our extended CLT \ref{GUE-clt} can be seen as the Bahadur-Rao improvement over the logarithmic limit implied by their MDP. The same speed for the Berry-Esseen bound is obtained. 

\begin{thm}[Extended central limit theorem for log-determinant of GUE]\ 
\label{GUE-clt}\\
For $y = o \left( \sqrt{\log n}\right),$ 
\begin{align*}
\P\left[X_n^{H} \geq y\sqrt{\frac{1}{2}\log\left(\frac{n}{2}\right)}\right]=\P\left[\mathrm
\mathcal{N}(0,1)\ge y\right]\left(1+o(1)\right).
\end{align*}
\end{thm}
\begin{thm}[Extended central limit theorem for log-determinants of $\b$-ensembles] \ \\
For $y = o \left( \sqrt{\log n}\right)$ and $i= L,G,J$
\begin{align*}
\P\left[X_n^{i,\b} \geq y\sqrt{\frac{2\log n}{\b}}\right]=\P\left[
\mathcal{N}(0,1)\ge y\right]\left(1+o(1)\right).
\end{align*}
\end{thm}

\begin{thm}[Precise moderate deviations for log-determinant of GUE]\ 
\label{GUE-pldp} \\ 
For $x > 0$, 
\begin{align*}
\P\left[X_n^{H} \ge    \frac{x}{2} \log \left(\frac{n}{2}\right)\right]=\frac{e^{-\frac{x^2}{4} \log \left(\frac{n}{2}\right) }}{x\sqrt{\pi \log \left(\frac{n}{2}\right)  }}\psi^{H}(x)\left(1+o(1)\right). 
\end{align*}
For $x\in\left(-1, 0\right)$, 
\begin{align*}
\P\left[X_n^{H} \le   \frac{x}{2} \log \left(\frac{n}{2}\right)\right]=\frac{e^{-\frac{x^2}{4} \log \left(\frac{n}{2}\right) }}{|x|\sqrt{\pi \log \left(\frac{n}{2}\right)  }}\psi^{H}(x)\left(1+o(1)\right) 
\end{align*}
\end{thm}
\begin{thm}[Precise moderate deviations for the log-determinants of $\b$-ensembles] \ \\
Let $i= L, G, J$. For $x > 0$, 
\begin{align*}
\P\left[X_n^{i, \b} \ge   \frac{2 x \log n }{\b}\right]=\frac{e^{-\frac{x^2 \log n}{\b}   }}{x\sqrt{\frac{4\pi \log n}{\b}  }}\psi^{i,\b}(x)\left(1+o(1)\right), 
\end{align*}
and for $x\in\left(-\bh, 0\right)$, 
\begin{align*}
\P\left[X_n^{i, \b} \le   \frac{2x\log n }{\b}\right]=\frac{e^{-\frac{x^2 \log n}{\b}   }}{|x|\sqrt{\frac{4\pi \log n}{\b}  }}\psi^{i,\b}(x)\left(1+o(1)\right)\,, 
\end{align*}
\end{thm}

Moreover, it is possible to show that we have a zone of control associated with the mod-Gaussian convergence of the sequences $\left(X_n^H\right)$ and $\left(X_n^{i,\b}\right),\ i=L,G,J$. This, in turn, implies the following additional quantitive estimates.

\begin{thm}[Speed of convergence for log-determinant of GUE]\ \label{GUE-speed}\\ 
\begin{align*}
d_{Kol}\left(X_n^{H}\sqrt{\frac{2}{\log(n/2)}}, \mathcal{N}(0,1)\right)\le C\left(\frac{1}{22},1,(3+K)e^{\frac{7}{4}+3K}\right)\sqrt{\frac{2}{\log(n/2)}}.
\end{align*}
where the constant $C(\cdot)$ is given in (\ref{defc})
and depends on the absolute constant $K$ of the big O term in Lemma \ref{main_theorem}.
\end{thm}
\begin{thm}[Speed of convergence for log-determinants of $\b$-ensembles]\ 
\label{determinants-speed}\\
 For $i= L, G, J$
\begin{align*}
d_{Kol}\left(X_n^{i,\b}\sqrt{\frac{\b}{2\log n}}, \mathcal{N}(0,1)\right)\le C\left(K_1^\b,1,\frac{1}{4K_2^\b}\right)\left(\frac{\b}{2\log n}\right)^{1/2}\,,
\end{align*}
where the constant $C(\cdot)$ is given in (\ref{defc}), $K_2^\b=\frac{8}{\b^2}+\frac{9}{\b}+4$ and $K_1^\b$ is a constant which depends on the  absolute constant $\tilde{K}$ of the big O term in Lemma \ref{main_theorem}.
\end{thm}

\begin{thm}[Local limit theorem for log-determinant of GUE]\ \label{GUE-llt}\\
Let $(a,b)$ be a fixed interval, with $a<b$. For every $-\half<\d<1$, 
\begin{align*}
\lim_{n\to\infty}\left(\frac{1}{2}\log\left(\frac{n}{2}\right)\right)^{\d+\half}\P\left(X_n^{H}\in \left(\frac{2}{\log(n/2)}\right)^\d (a,b)\right)=\frac{b-a}{\sqrt{2\pi}}.
\end{align*}
\end{thm}

\begin{thm}[Local limit theorem for log-determinants of $\b$-ensembles]\ 
\label{determinants-llt} \\
Let $(a,b)$ be a fixed interval, with $a<b$. For every $-\half<\d<1$, and $i = L, G,J$
\begin{align*}
\lim_{n\to\infty}\left(\frac{2\log n}{\b}\right)^{\d+\half}\P\left(X_n^{i,\b}\in \left(\frac{\b}{2\log n}\right)^\d (a,b)\right)=\frac{b-a}{\sqrt{2\pi}},\\
\end{align*}
\end{thm}

\subsection{Random characteristic polynomials}
\label{subs:characteristicpolyresults}
The  following lemma represents the analogous of lemmas \ref{GUE-lemma} and \ref{Determinant-lemma} in the circular case. 
\begin{lemma}\label{JacCirc-lemma}
The cumulant generating functions of the $\b$-circular Jacobi characteristic polynomial satisfies
\begin{align*}
\log \E\left[e^{z\log |Z_n^{\b,\d}|}\right]& = z \frac{\d+\bar{\d}}{\b} +\frac{z^2}{2\b}\log n+ \Upsilon\left(1+\d-\bh\right)-  \Upsilon\left(1+\d+\frac{z}{2}-\bh\right)\\
&\quad + \Upsilon\left(1+\bar{\d}-\bh\right) -\Upsilon\left(1+\d+\bar{\d}-\bh\right) \\
&\quad - \Upsilon\left(1+\bar{\d}+\frac{z}{2}-\bh\right)+ \Upsilon\left(1+\d+\bar{\d}+z-\bh\right) +o(1),
\end{align*}
locally uniformly on the band $S_{\frac{1}{3}}$. In particular for the $\b$-circular ensemble, i.e. for $\d=0$, we have
\begin{align*}
\log \E\left[e^{z\log |Z_n^{\b}|}\right]& = \frac{z^2}{2\b}\log n+ \Upsilon\left(1-\bh\right)- 2 \Upsilon\left(1+\frac{z}{2}-\bh\right)\\
&\quad + \Upsilon\left(1+z-\bh\right) +o(1),
\end{align*}
\end{lemma}

\begin{thm}[Mod-Gaussian convergence for the log-characteristic polynomial of $\b$-circular Jacobi ensembles]\ \\
Let 
\begin{align*}
&X_n^{CJ, \b,\d} := \log |Z_n^{\b,\d}| - \frac{\d+\bar{\d}}{\b}\log n .
\end{align*}
Then as $n \to \infty,$ $\left( X_n^{CJ, \b,\d}\right)_{n \in \N}$ converges mod-Gaussian on the strip $S_{\frac{1}{3}}$  with parameters $t_n^{\b} = \frac{\log n}{2\b}$
and with limiting functions $\psi^{CJ,\b,\d}$ which satisfies
\begin{align*}
\log \psi^{CJ,\b, \d}(z) =& \Upsilon(1+\d-\bh) -\Upsilon(1+\d+\frac{z}{2}-\bh) + \Upsilon(1+\bar{\d}-\bh) -\Upsilon(1+\bar{\d}+\frac{z}{2}-\bh)\\&-
 \Upsilon(1+\d+\bar{\d}-\bh) + \Upsilon(1+\d+\bar{\d}+z-\bh). 
\end{align*}
\end{thm}
\begin{thm}[Precise moderate deviations for the log-characteristic polynomial of the $\b$-circular Jacobi ensemble] \ \\
For $x > 0$, 
\begin{align*}
\P\left[X_n^{CJ, \b,\d} \ge   \frac{x \log n }{2\b}\right]=\frac{e^{-\frac{x^2 \log n}{4 \b}   }}{x\sqrt{\frac{\pi \log n}{\b}  }}\psi^{CJ,\b,\d}(x)\left(1+o(1)\right) 
\end{align*}
and for $x \in \left(-1/3, 0 \right) $ 
\begin{align*}
\P\left[X_n^{CJ, \b,\d} \le   \frac{x \log n }{2\b}\right]=\frac{e^{-\frac{x^2 \log n}{4 \b}   }}{|x|\sqrt{\frac{\pi \log n}{\b}  }}\psi^{CJ,\b,\d}(x)\left(1+o(1)\right).
\end{align*}
\end{thm}
\begin{thm}[Extended central limit theorem for  log-characteristic polynomial of the $\b$-circular Jacobi ensemble] \ \\
For $y = o \left( \sqrt{\log n}\right)$
\begin{align*}
\P\left[X_n^{CJ,\b,\d} \geq y\sqrt{\frac{\log n}{2\b}}\right]=\P\left[\mathcal{N}(0,1)\ge y\right]\left(1+o(1)\right).
\end{align*}
\end{thm}

\begin{thm}[Speed of convergence for log-characteristic polynomial of the $\b$-circular Jacobi ensemble]\
\label{charactpoly-speed} \\ 
\begin{align*}
d_{Kol}\left(X_n^{CJ,\b,\d}\sqrt{\frac{2\b}{\log n}}, \mathcal{N}(0,1)\right)\le C\left(K_1^\b,1,\frac{1}{4K_2^\b}\right)\left(\frac{\b}{2\log n}\right)^{1/2}\,,
\end{align*}
where the constant $C(\cdot)$ is given in (\ref{defc}), $K_2^\b=\frac{8}{\b^2}+\frac{9}{\b}+4$ and   $K_1^\b$ is a constant which depends on the  absolute implied constant $\tilde{K}$ of the big O term in Lemma \ref{main_theorem}.
\end{thm}

\begin{thm}[Local limit theorem for log-characteristic polynomial of the $\b$-circular Jacobi ensemble]\
\label{charactpoly-llt}\\ 
Let $(a,b)$ be a fixed interval, with $a<b$. 
\begin{align*}
\lim_{n\to\infty}\left(\frac{\log n}{2\b}\right)^{\d+\half}\P\left(X_n^{CJ,\b,\d}\in \left(\frac{2\b}{\log n}\right)^\d (a,b)\right)=\frac{b-a}{\sqrt{2\pi}}.
\end{align*}
\end{thm}

All the above stated limiting theorems hold for the $\b$-circular ensemble considering $\d=0$.

\section{A key asymptotic expansion}
\label{section:asymptotics}
In order to show mod-Gaussian convergence for the log-determinants of random ensembles or for log-characteristic polynomials, we need to establish the asymptotic behavior of their moment generating functions. In most cases, it boils down to estimate the following product of the ratios of Gamma functions
\begin{equation}\label{ratio}
\prod_{k=1}^n \frac{\G\left(\bh k+z \right)}{\G\left(\bh k \right) },
\end{equation}
where $\bh > 0$ and $z$ belongs to the strip $S_{\bh}$. 
Note that, instead of having the convergence of these moment generating functions on the entire strip $S_{\bh}$, it is sufficient to prove it on a domain that becomes larger and larger and approaches the strip as $n$ goes to infinity.
This justifies the domain
$\lbrace z: z \in S_{\bh}, |z| <\frac{\b}{8} n^{1/6} \rbrace$ used in the subsequent theorems.\\
We first present a key asymptotical result for (\ref{ratio}) for every real $\b>0$, we shall use in order to deduce the expansions in lemmas \ref{GUE-lemma}, \ref{Determinant-lemma} and \ref{JacCirc-lemma}. The main problem is addressed using the Abel-Plana formula (Theorem \ref{Abel-Plana}) which allows to evaluate some specific even non-convergent sums. Indeed, $z$ being unbounded, we can not simplify these sums using other approaches (e.g. the Taylor expansion).

\begin{thm} \label{main_theorem} For all $n\ge 1$ and any $z\in S_{\bh}$ with $|z| < \frac{\b}{8} n^{1/6}$, we have 
\begin{align*}
\log\left(\prod_{k=1}^n\frac{\G\left(\bh k+z\right)}{\G\left(\bh k\right)}\right)&=z\left(\left(\half-\frac{1}{\b}\right)\log n+n\log\left(\bh n\right)-n\right)\\
&\quad +\frac{z^2}{\b}\log n +\Upsilon(z)+O\left(\frac{|z|+|z|^2+|z|^3}{n} \right),
\end{align*}
for some absolute implied constant, where
\begin{align}\label{upsilon}
\Upsilon(z)&=\bh \log G\left(\frac{2z}{\b}+1\right)-\left(z-\half\right)\log\G\left(\frac{2z}{\b}+1\right)\\
&\nonumber \quad +\int_0^\infty\frac{\varphi(s)\left(\ex^{-sz}-1\right)}{\ex^{s\bh}-1}ds+\frac{z^2}{\b}+\frac{z}{2}, 
\end{align}
with
\begin{equation*}
\varphi(s)=\left( \frac{1}{2} - \frac{1}{s} + \frac{1}{\ex^s-1}\right) \frac{1}{s} 
\end{equation*}
and $G$ denotes the Barnes $G$-function.
\end{thm}                                    
        
\begin{proof}
For simplicity we set $\b'=\bh$. To obtain the asymptotic of 
 $\sum_{k=1}^n \left[l(\b'k+z)-l(\b'k)\right]$ as $n$ goes to infinity, we will first use  the Binet's formula (\ref{Binet1}).
Then,
\begin{align}
\label{start}
\sum_{k=1}^n \ell\left(\b'k+z\right) - \ell\left(\b' k\right))&= 
S_1 (n,z) -nz + S_2 (n,z)
\end{align}
where 
\begin{align}
S_1(n,z) := \sum_{k=1}^n \left[\left(\b'k+z-\half\right)\log\left(\b'k+z\right)-
\left(\b'k-\half\right)\log\left(\b'k\right)\right]
\end{align}
and 
\begin{align}
\nonumber
S_2(n,z) &:= \sum_{k=1}^n\int_0^\infty \varphi(s)\left(\ex^{-s\left(\b'k+z\right)}-\ex^{-s\b'k}\right)ds\\
\nonumber
&=\int_0^\infty \varphi(s)\left(\ex^{-sz}-1\right)\sum_{k=1}^n \ex^{-s\b'k}ds\\
&=\int_0^\infty\frac{\varphi(s)\left(\ex^{-sz}-1\right)}{\ex^{s\b'}-1}ds-\int_0^\infty \frac{\varphi(s)\left(\ex^{-sz}-1\right)}{\ex^{s\b'}-1}\ex^{-s\b'n}ds.
\end{align}
Using the inequalities $\ex^x - 1 \geq  x$ and $\left|\ex^z-1\right|\le |z|\ex^{|z|}$ for any $x\in [0, \infty)$, respectively $z\in\C$, we obtain
\begin{align*}
\left| \int_0^\infty \varphi(s)\frac{\ex^{-sz}-1}{\ex^{s\b'}-1}\ex^{-s\b'n}ds \right| &\leq \frac{1}{12}
\int_0^\infty \frac{\left|\ex^{-sz}-1\right|}{\left|\ex^{s\b'}-1\right|}\ex^{-s\b'n}ds\\
&\le 
\frac{1}{12}\int_0^\infty \frac{s\left|z\right| \ex^{s|z|}}{s\b'}\ex^{-s\b'n}ds
\le \frac{1}{6\left(\b'\right)^2}\frac{|z|}{n},
\end{align*}
as soon as $|z| \leq \beta'n/2$, which is widely compatible with the assumptions of Theorem  \ref{main_theorem}.
 We get
\begin{equation}
\label{s2}
S_2(n,z) = \int_0^\infty\frac{\varphi(s)\left(\ex^{-sz}-1\right)}{\ex^{s\b'}-1}ds + O\left(\frac{|z|}{n}\right)
\end{equation}
with an absolute implied constant. \\
It remains to estimate the term $S_1(n,z)$.
\begin{align}\nonumber
S_1(n,z)&=\sum_{k=1}^n \left[\left(\b'k+z-\half\right)\log\left(\b'k+z\right)-
\left(\b'k-\half\right)\log\left(\b'k\right)\right]\\
\nonumber
&=\sum_{k=1}^n \left(\b'k+z-\half\right)\log\left(1+\frac{z}{\b'k}\right)+z\sum_{k=1}^n\log(\b'k)\\
\label{s1-comput}
&=z\log\left((\b')^n n!\right)+\sum_{k=0}^{n-1} f(k)\,,
\end{align}
where
\begin{equation}
\label{f}
f(s) =\left(\b'(s+1)+z-\half\right)\log\left(1+\frac{z}{\b'(s+1)}\right).
\end{equation}
Note that, since $z$ is growing to infinity with $n$, it is not possible to Taylor-expand the term $\log\left(1+\frac{z}{\b'(k+1)}\right)$ for all $k$. Here is where the Abel-Plana formula comes into play. It is an easy exercise to check that the conditions of Theorem \ref{Abel-Plana} are satisfied. We leave the details to the reader. 
Then by (\ref{AP})
\begin{align}
\label{i1-rn}
\sum _{k=0}^{n-1} f(k) & = I_1(n,z) + R_n (z)
\end{align}
where
\begin{align}
I_1(n,z) &:= \int_1^{n+1}\left(\b's+z-\half\right)\log\left(1+\frac{z}{\b's}\right)ds \nonumber \\
&-\half \left(\b'(n+1)+z-\half \right) \log\left( 1+\frac{z}{\b'(n+1)}\right),
\end{align} 
and 
where $R_n(z)$ is the reminder coming from the Abel-Plana formula, 
\begin{equation}
\label{rn}
R_n(z):=\half \left(\b'+z-\half\right)\log\left(1+\frac{z}{\b'}\right)+ I_2(z) -I_3(n,z),
\end{equation}
with
\begin{align}
\nonumber
I_2(z)& =:i\int_0^{\infty}\frac
{f(is) - f(-is)}{e^{2\pi s}-1}ds,\\
\label{i2-i3}
I_3(n,z) &=: i\int_0^{\infty}\frac{f(n+is)-f(n-is)}{e^{2\pi s}-1}ds.
\end{align}
We start showing that
\begin{equation*}
I_3(n,z)=O\left(\frac{|z|+|z|^2
}{n}\right).
\end{equation*}
for some absolute implied constant, eventually depending on $\b'$. 
First,
\begin{align}
\nonumber
\left|I_3(n,z)\right| \leq  \int_0^{\infty}\frac{ \b's}{e^{2\pi s}-1}\left|\log\left(1+\frac{z}{\b'(1+n+is)}\right)+\log\left(1+\frac{z}{\b'(1+n-is)} \right)\right|ds \\ \label{i3}
+ \int_0^{\infty}\frac{\left|\b'(1+n)+z-\half\right|}{e^{2\pi s}-1} \left| \log\left(1+\frac{z}{\b'(1+n+is)}\right)-\log\left(1+\frac{z}{\b'(1+n-is)}\right) \right|ds. 
\end{align}
We proceed  estimating both above integrals, using the formula 
\begin{equation}
\label{deflog}
\log (1 + \zeta) = \int_0^1\frac{\zeta}{t\zeta + 1} dt,
\end{equation}
which is valid for $\zeta \in \mathbb C \setminus(-\infty, -1]$.
\\
For the first integral (\ref{i3}), note that, from (\ref{deflog})
\begin{equation*}
|\log(1+\omega)| \leq \frac{|\omega|}{a} \log \frac{1}{1-a},
\end{equation*}
as soon as $|\omega| \leq a<1$.
Therefore since
\begin{align*}
\frac{|z|}{|\b'(1+n+is)|}\le \frac{|z|}{\b' n}\le \frac{1}{4},
\end{align*}
it follows
\begin{equation*}
\left|\log\left(1+\frac{z}{\b'(1+n+is)}\right)+\log\left(1+\frac{z}{\b'(1+n-is)} \right)\right|=O\left(\frac{|z|}{|\b'(1+n+is)|}\right)=O\left(\frac{|z|}{n}\right)
\end{equation*}
for some absolute implied constant and all $n\ge 1$.\\
For the second integral, 
using (\ref{deflog}) again
we have, for $|u|, |v| \leq a < 1$
\begin{equation*}
\left| \log(1+u) - \log (1+v)\right| \leq \frac{|u-v|}{1-a},\end{equation*}
so that  
\begin{equation*}
\left| \log\left(1+\frac{z}{\b'(1+n+is)}\right)-\log\left(1+\frac{z}{\b'(1+n-is)}\right) \right|=O\left(\frac{|z|s}{(\b')(1+n)^2}\right),
\end{equation*}
for some absolute implied constant and all $n\ge 1$ as well.\\
Gathering all estimates we obtain
\begin{equation}
\label{i3-bound}
\left|I_3(n,z)\right| \leq O\left(\frac{|z|+|z|^2}{n}\right).
\end{equation}
The expression $I_1(n,z)$ can be computed explicitly, using integration by parts: 
\begin{align*}
I_1(n,z)&= \int_1^{n+1}\left(\b' s + z- \half\right) \log \left(1 + \frac{z}{\b's}\right) ds \\
&-\half\left(\b'(n+1) +z-\half\right)\log \left(1 + \frac{z}{\b'(n+1)}\right)\\
&= \left[  \left(\b'\frac{s^2}{2} +(z-\half) s\right)\log \left(1 + \frac{z}{\b's}\right)\right]_1^{n+1} +\frac{z}{2}\int_1^{n+1}\frac{\b's + 2z-1}{z+\b's} ds\\
&-\half\left(\b'(n+1) +z-\half\right)\log \left(1 + \frac{z}{\b'(n+1)}\right)\\
\end{align*}
and since
\begin{align*}\int_1^{n+1}\frac{\b's + 2z-1}{z+\b's} ds &= n + \frac{z-1}{\b'} \left[\log (s+ \frac{z}{\b'})\right]_1^{n+1}\\
&= n + \frac{z-1}{\b'}\log (n+1)\\
& + \frac{z-1}{\b'}\log \left(1 + \frac{z}{\b'(n+1)}\right)- \frac{z-1}{\b'}\log(1 + \frac{z}{\b'})\,,
\end{align*}
we get
\begin{align*}
I_1(n,z)&= \left(\b'\frac{n(n+1)}{2} + (z-\half) (n +\half) + \frac{z(z-1)}{2\b'}\right) \log \left(1 + \frac{z}{\b'(n+1)}\right)\\
&- \left(\frac{\b'}{2} + z - \half + \frac{z(z-1)}{2\b'}\right) \log \left(1 + \frac{z}{\b'}\right)\\
&+ \frac{nz}{2} + \frac{z(z-1)}{2\b'}\log (n+1),.\
\end{align*}
In the last step, we expand  the logarithm, and conclude:
\begin{align*}
I_1(n,z)&= nz +  \frac{z(z-1)}{2\b'} \log n   -\frac{z}{2\beta'}+ \frac{3z^2}{4\b'} \\
&+ \left(-\frac{\b'}{2} - z + \half - \frac{z(z-1)}{2\b'}\right) \log \left(1 + \frac{z}{\b'}\right) + O\left( \frac{|z| + |z|^2+ |z|^3}{n}\right)\,.
\end{align*}
Our last goal is to write $I_2(z)$ in terms of the Barnes and Gamma functions. 
First, for $f$ defined in (\ref{f}), an elementary computation leads to
\begin{align*}
f(is) - f(-is) &= i\b' s \log \left(1 + \left(1+ \frac{z}{\b'}\right)^2s^{-2}\right) -i \b' s \log \left(1 + s^{-2}\right)\\
&\quad  +2i \left(\b' + z -\half \right) \left(\arctan \frac{s}{1 + \frac{z}{\b'}} - \arctan s\right),
\end{align*}
and then,
\begin{align*}
I_2(z)&= i \int_0^\infty \frac{f(is) - f(-is)}{e^{2\pi s} -1} ds \\
&= -\b' \int_0^\infty  \log \left(1 + \left(1+ \frac{z}{\b'}\right)^2s^{-2}\right)  \frac{s\!\ ds}{e^{2\pi s} -1} +   
\b' \int_0^\infty  \log \left(1 + s^{-2}\right)\frac
{s\!\ ds}
{e^{2\pi s} -1} \\
&\quad  - \left(\b' + z -\half\right)  \int_0^\infty  \left(\arctan \frac{s}{1 + \frac{z}{\b'}}\right) \frac{ds}{e^{2\pi s} -1} + (\b' + z -\half)  \int_0^\infty \arctan s   \frac{ds}{e^{2\pi s} -1}.
\end{align*}
Using the representation (\ref{logG}) and the second Binet formula (\ref{Binet2}) we get
\begin{align}
\nonumber
I_2(z)&=\b'\log G\left(\frac{z}{\b'}+1\right)-\left(z-\half\right)\log\G\left(\frac{z}{\b'}+1\right)\\
\label{i2-comput}
&\quad +\log\left(\frac{z}{\b'}+1\right)\left[\frac{z^2}{2\b'}-\frac{z}{2\b'}+\frac{z}{2}-\frac{1}{4}\right]-\frac{z^2}{4\b'}+\frac{z}{2\b'}+\frac{z}{2}-\frac{z}{2}\log 2\pi\,.
\end{align}
Gathering (\ref{rn}), (\ref{i2-i3}), (\ref{i3-bound}) and (\ref{i2-comput}) we arrive at
\begin{align}
\nonumber
R_n (z) &= \half \left(\b'+z-\half\right)\log\left(1+\frac{z}{\b'}\right)\\
\nonumber
&\quad +\b'\log G\left(\frac{z}{\b'}+1\right)-\left(z-\half\right)\log\G\left(\frac{z}{\b'}+1\right)\\
\nonumber
&\quad +\log\left(\frac{z}{\b'}+1\right)\left[\frac{z^2}{2\b'}-\frac{z}{2\b'}+\frac{z}{2}-\frac{1}{4}\right]-\frac{z^2}{4\b'}+\frac{z}{2\b'}+\frac{z}{2}-\frac{z}{2}\log 2\pi\\
\label{rn-bound}
&\quad + O\left(\frac{|z|+|z|^2}{n}\right)\,, 
\end{align}
and with (\ref{start}), (\ref{s2}),  (\ref{s1-comput}) (\ref{i1-rn}) and  (\ref{rn-bound}) we may conclude
\begin{align*}
\log\left(\prod_{k=1}^n\frac{\G(\b'k+z)}{\G(\b'k)}\right)&=z\log\left((\b')^n n!\right)+z\left(n-\frac{1}{\b}\log n\right)+\frac{z^2}{2\b'}\log n\\
 &\quad +\log\left(1+\frac{z}{\b'}\right)\left[-\frac{z^2}{2\b'}-z+\frac{z}{2\b'}+\half-\frac{\b'}{2}\right]+\frac{3z^2}{4\b'}-\frac{z}{2\b'}\\
 &\quad +\half \left(\b'+z-\half\right)\log\left(1+\frac{z}{\b'}\right)\\
 &\quad +\b'\log G\left(\frac{z}{\b'}+1\right)-\left(z-\half\right)\log\G\left(\frac{z}{\b'}+1\right)\\
&\quad +\log\left(\frac{z}{\b'}+1\right)\left[\frac{z^2}{2\b'}-\frac{z}{2\b'}+\frac{z}{2}-\frac{1}{4}\right]-\frac{z^2}{4\b'}+\frac{z}{2\b'}+\frac{z}{2}\\
&\quad -\frac{z}{2}\log 2\pi -nz+\int_0^\infty\frac{\varphi(s)\left(e^{-sz}-1\right)}{e^{s\b'}-1}ds+O\left(\frac{|z|+|z|^2+|z|^3}{n} \right)\\
&=z\left(\left(\half-\frac{1}{\b}\right)\log n+n\log(\b'n)-n\right)+\frac{z^2}{2\b'}\log n\\
&\quad +\b'\log G\left(\frac{z}{\b'}+1\right)-\left(z-\half\right)\log\G\left(\frac{z}{\b'}+1\right)\\
&\quad +\int_0^\infty\frac{\varphi(s)\left(e^{-sz}-1\right)}{e^{s\b'}-1}ds+\frac{z^2}{2\b'}+\frac{z}{2}+O\left(\frac{|z|+|z|^2+|z|^3}{n} \right),
\end{align*}
which ends the proof.
\end{proof}

\begin{rmk} In the same way as for Theorem \ref{main_theorem}, it is possible to prove the following slightly more general statement:
\begin{thm}\label{main_theorem2} 
Let $\delta\in\C$ be a fixed complex number with $\Re(\d)>-\bh$ and consider $n_\d\in\N$ big enough such that $|\d| < \frac{\b}{8} n_\d^{1/6}$. Then for all $n\ge n_\d$ and any $z\in\C$, such that $\Re(\d+z)>-\bh$ and $|z| < \frac{\b}{8} n^{1/6}$, we have 
\begin{align*}
\log\left(\prod_{k=1}^n\frac{\G\left(\bh k+\d+z\right)}{\G\left(\bh k
+\d\right)}\right)&=z\left(\left(\half-\frac{1}{\b}+\frac{2\d}{\b}\right)\log n+n\log\left(\bh n\right)-n\right)\\
&\quad +\frac{z^2}{\b}\log n +\Upsilon_\d(z)+O\left(\frac{|z|+|z|^2+|z|^3}{n} \right),
\end{align*}
for some absolute implied constant, with $\Upsilon_\d(z):=\Upsilon(z+\d)-\Upsilon(\d)$,
%
where the function $\Upsilon$ is defined as in Theorem \ref{main_theorem}.
\end{thm}   
\end{rmk}

\section{Proofs of lemmas \ref{GUE-lemma}, \ref{Determinant-lemma}, \ref{JacCirc-lemma} and of theorems \ref{GUE-speed}, \ref{determinants-speed}, \ref{GUE-llt}, \ref{determinants-llt},  \ref{charactpoly-speed} and \ref{charactpoly-llt} }
\label{section:proofs}

\begin{proof}[Proof of lemma \ref{GUE-lemma}]
In the following, we shall find an asymptotic for the Laplace transform of $\log\left|\det W^H_n\right|$ when $n$ is odd. We leave to the reader to check that $n$ even the same expansion holds.
Note that, for $n$ odd, (\ref{Her-Mellin}) is equivalent to
\begin{equation}\label{Her-Mellinodd}
\E\left|\det W^H_n\right|^z=2^{\frac{nz}{2}}\frac{\G\left(\frac{z+1}{2}\right)}{\G\left(\frac{1}{2}\right)}\prod_{k=1}^{\frac{n-1}{2}}\frac{\G\left(\frac{z+1}{2}+k\right)^2}{\G\left(\frac{1}{2}+k\right)^2}.
\end{equation}
Then, by Lemma \ref{rational}.\ref{integer} we obtain 
\begin{equation*}
\log\E\left|\det W^H_n\right|^z= z \mu^{H} _{ n}  + \frac{z^2}{4} \log\left(\frac{n}{2}\right)+ \Upsilon^H(z)+o(1),
\end{equation*}
locally uniformly on $S_1$. \\
Equivalently, a direct method would apply to (\ref{Her-Mellinodd}) the finite product formula satisfied by the Barnes function (\ref{finiteG}) and get that
\begin{align}
\nonumber
\E\left|\det W^H_n\right|^z&=2^{\frac{nz}{2}}\frac{\G\left(\frac{z+1}{2}\right)}{\G\left(\frac{1}{2}\right)}\left(\frac{G\left(\frac{z+1}{2}+1+\frac{n-1}{2}\right)}{G\left(\frac{z+1}{2}+1\right)}\frac{G\left(\frac{1}{2}+1\right)}{G\left(\frac{1}{2}+1+\frac{n-1}{2}\right)}\right)^2\\
\label{first}
&=\frac{\G\left(\frac{z+1}{2}\right)}{\G\left(\frac{1}{2}\right)}\frac{G\left(\frac{3}{2}\right)^2}{G\left(\frac{z+1}{2}+1\right)^2}2^{\frac{nz}{2}}\left(\frac{G\left(\frac{z}{2}+\frac{n}{2}+1\right)}{G\left(\frac{n}{2}+1\right)}\right)^2.
\end{align}
Then, to manage the factor which depends on $n$ in (\ref{first}), one could use the estimation of Proposition 17 in  \cite{KN12}, which holds true for $|\zeta| \leq \frac{1}{2}p^{1/6}$ and gives
\begin{equation*}
\frac{G(1+\zeta + p)}{G(1+p)} = (2\pi)^{\zeta/2}e^{-(p+1)\zeta} (1+p)^{\zeta^2/2 + p\zeta} S_p(\zeta)
\end{equation*}
with
\begin{equation*}\log S_p(\zeta) = O\left(\frac{|\zeta|^2 + |\zeta|^3}{p}\right)
\end{equation*}
with some absolute implied constant. This yields the same expansion as the one obtained above.
\end{proof}

\begin{proof}[Proof of lemma \ref{Determinant-lemma}]
This proof follows directly from the Mellin transforms of the determinants given by  formulas (\ref{Lag-Mellin}),  (\ref{Gram-Mellin}) and (\ref{Jac-Mellin}) and from Theorem \ref{main_theorem}. Additionally, for the  Uniform Gram ensemble, we used the following asymptotic expansion which is obtained by means of formula (\ref{Binet1}):
\begin{equation*}
l\left( \b' n + z\right) - l\left( \b' n\right) = z \left(  - \frac{1}{2 \b' n} + \log \b'+ \log n\right) + O\left(\frac{|z|+|z|^2+|z|^3}{n}\right)
\end{equation*}
for any $z\in S_{\frac{\b}{2}}$ with $|z| <\frac{\b}{8} n^{1/6}$.
\end{proof}

\begin{proof}[Proof of lemma \ref{JacCirc-lemma}] For the Circular Jacobi ensemble we simply apply Theorem \ref{main_theorem} for all terms appearing in the 
 (\ref{JacCirc-Mellin}).
\end{proof}

\begin{proof}[Proof of theorems \ref{GUE-speed} and \ref{GUE-llt}]
In order to establish both theorems, it is enough to check that mod-Gaussian convergence happens with a zone of control. Therefore we check if condition \ref{Z1} is satisfied for some parameters $\g>-\half$ and $\nu, w, D, K_1$ and $K_2$. In order to find an upper bound for $\left|\psi_n^H(i\xi)-1\right|$, we first compute an upper bound for $\left|\psi^H(i\xi)-1\right|$. The reason for this will be clear from the proof.
\vskip 10pt
\noindent\underline{1) Bound for  $\left|\psi^H(i\xi)-1\right|.$}
Using the inequality 
\begin{equation}
\label{fundineq}\left|\ex^z-1\right|\le |z|\ex^{|z|} \ ,\  (z\in\C)\;,
\end{equation} we obtain
\begin{equation}
\label{program}
\left|\psi^H(i\xi)-1\right|=\left|e^{\Upsilon^H(i\xi)}-1\right| \le \left|\Upsilon^H(i\xi)\right| e^{\left|\Upsilon^H(i\xi)\right|}
\end{equation}
where in $\Upsilon^H$ the principal branch of the logarithm is considered. Therefore, in order to achive our goal, it is sufficient to find an upper bound for $\left|\Upsilon^H(i\xi)\right|$. 
Set
\begin{equation}
\label{Upsilonagain}
f(\xi) := \Upsilon^H(i\xi) = -\ell 
\left(\frac{i\xi+1}{2}\right)-2\log G\left(\frac{i\xi+1}{2}\right).
\end{equation}
The function $f$ satisfies $f(0)=0$. Thus, by Theorem 5.19 in \cite{R76} 
\begin{equation}
\label{meanvalue}
|f(\xi)| \leq |\xi|\sup_{t \in (0,\xi)}\left|f'(t)\right|.
\end{equation}
But 
\begin{equation*}f'(t) = \frac{i}{2}\Psi\left(\frac{i t+1}{2}\right)+i\frac{G'\left(\frac{it+1}{2}\right)}{G\left(\frac{it+1}{2}\right)}
\end{equation*}
and using the formula  (\ref{der-Barnes}) together with (\ref{digamma}), we get
\begin{align}
\nonumber f'(t) &= \frac{it}{2}\Psi\left(\frac{i t+1}{2}\right)-\frac{it}{2}+\half\log 2\pi\\
\label{derivative}
&= \frac{it}{2}\log\left(\frac{it+1}{2}\right)-\frac{it}{2}\int_0^\infty e^{-s\left(\frac{it+1}{2}\right)}\left(s\phi(s)+\half\right)-\frac{it}{2}+\half\log 2\pi\,.
\end{align}
Using the elementary inequality
\begin{equation}
\label{ineqlog}\left|\log (1+ it)\right| \leq |t| \,,\end{equation}
(make $\zeta = it$ in  (\ref{deflog}))
and (\ref{digamma_bound}) we get
%
\begin{align}
\label{ee}
\left|\Upsilon^H(i\xi)\right|
&\le \frac{|\xi|^3}{2}+3|\xi|^2+|\xi|\\
\label{ff}|
&\le 3|\xi|e^{|\xi|}\,,
\end{align}
so that
\[|\Upsilon^H(i\xi)e^{|\Upsilon^H(i\xi)|} \leq 3|\xi|e^{\frac{|\xi|^3}{2}+3|\xi|^2+2|\xi|}\,.\]
Now, it is clear that there exists $k > 0$ such that, for every $x \geq 0$
\begin{align}
\label{trinom}
\frac{x^3}{2} + 3x^2 +2x \leq x^3 +k\,,
\end{align} 
and then
\begin{align}
\label{intermH}
|\Upsilon^H(i\xi)e^{|\Upsilon^H(i\xi)|} \leq 3 e^k |\xi|e^{|\xi|^3}\,.
\end{align}
Using (\ref{program}), we conclude
\begin{align*}
\left|\psi^H(i\xi)-1\right|
&\le 3 e^k |\xi|e^{|\xi|^3}\,.
\end{align*}
\vskip 10pt

\noindent\underline{2) Bound for $|\psi_n(i\xi) -1|.$}
From the proof of theorems \ref{GUE-lemma} and \ref{main_theorem},
\begin{equation*}
\psi_n^H(z)=e^{\Upsilon^H(z)+ u_n(z)}.
\end{equation*}
where
\begin{equation*}
u_n(z) := O\left(\frac{|z| + |z|^2+ |z|^3}{n}\right)
\end{equation*}
as soon as $|z| \leq \half n^{1/6}$. It means that  there exist a constant $K>0$  such that for every $n\ge 1$ integer, and  $|\xi| \leq \half n^{1/6}$
\begin{equation*}
|u_n(i\xi)| \leq K \frac{|\xi|+|\xi|^2+|\xi|^3}{n}.
\end{equation*}
Actually, we have 
\begin{equation*}
|u_n(i\xi)|  \le K|\xi|e^{|\xi|}
\end{equation*}
and also, since $|\xi|\le\half n^{1/6}$,
\begin{equation*}
|u_n(i\xi)|\le K,
\end{equation*}
Using (\ref{fundineq}) again we have
\[|\psi_n^H(i\xi) -1| \leq K\left(|\Upsilon (i\xi)| + K  |\xi| e^{ |\xi|}\right) e^{|\Upsilon (i\xi)|} \]
We then plug successively (\ref{ff}),(\ref{ee}) and (\ref{trinom}) and obtain 
\begin{align*}\left|\psi^H_n(i\xi)-1\right|&\le K(3+K)e^k  |\xi|e^{|\xi|^3}
\end{align*}
Therefore the sequence $\left(X^H_n\right)$ converges mod-Gaussian with zone of control $\left[-Dt_n^H, Dt_n^H\right]$ and index of control $(1,3)$. In particular the parameters could be choosen as $K_1 = K(3+K)e^k$, $K_2 = 1$ and $D= 1/4$.
\\
Theorem \ref{GUE-speed} follows considering $\g=\min\{1,\frac{\nu-1}{2}\}=0.$
\end{proof}

\begin{proof}[Proof of theorems \ref{determinants-speed} and \ref{determinants-llt}]
It is possible to follow the same scheme of proof as for the GUE case. Moreover, since the techniques and the methods are similar for all the eigenvalues statistics $X_n^{i,\b},\ i=L,J,G$, we prove the theorem only for the sequence $\left(X_n^{L,\b}\right)_{n\in\N}$ to avoid repetitions. \\
In order to reach this goal, we need to check whether the condition \ref{Z1} is satisfied, proceeding as we did for the complex Gaussian ensemble. The condition \ref{Z2} can then be forced.
\vskip 10pt
\noindent\underline{1) Bound for  $\left|\psi^{L,\b}(i\xi)-1\right|$}.
Using again the inequality (\ref{fundineq}) 
we obtain
\begin{equation}
\label{programL}
\left|\psi^{L,\b}(i\xi)-1\right|=\left|  e^{\Upsilon(i\xi)} - 1 \right| \leq \left|\Upsilon(i\xi) \right| e^{\left|\Upsilon(i\xi) \right|}
\end{equation}
and is therefore sufficient to find an upper bound for $\left| \Upsilon(i\xi) \right|.$ Recall that, 
\begin{align*}
\Upsilon(i\xi) & = -\frac{\xi^2}{\b}+\frac{i \xi}{2}+f(\xi)+\int_0^\infty\frac{\varphi(s)\left(e^{-si \xi}-1\right)}{e^{s\bh}-1}ds.
\end{align*}
where $f(\xi) := \bh \log G\left(\frac{i 2\xi}{\b}+1\right)-\left(i\xi-\half\right)\ell\left(\frac{i2\xi}{\b}+1\right).$ Thus, using (\ref{12}),
\begin{align*}
\left|\Upsilon(i\xi)\right|&\le \frac{|\xi|^2}{\b}+\frac{|\xi|}{2}+\left|f(\xi)\right|+\frac{|\xi|\pi^2}{18\b^2}.
\end{align*}
By Theorem 5.19 in \cite{R76}, and since $|f(0) = 0$,
\begin{align*}
\left| f(\xi)\right| \leq |\xi| \sup_{t \in (0,\xi)} \left|f'(t)\right|\,.
\end{align*}
Now
\[f'(t) = \frac{ G'\left(\frac{2i t}{\b}+1\right)}{G\left(\frac{2i t}{\b}+1\right)} - \ell \left( \frac{2it}{\b} +1\right) - \frac{2}{\b} \left( it - \half\right) \Psi\left(\frac{2it}{\b}+1\right)\,.\]
Using successively (\ref{ineqlog}) 
 and for $a_1, \dots, c_3$ some positive constants
\[\left|\ell(1 +it)\right| \leq a_1t^2 +b_1 |t| + c_1\,\]
(from (\ref{Binet1})), 
\[\left|\Psi(1+it)\right| \leq b_2 |t| + c_2\,,\]
(from (\ref{digamma}), (\ref{digamma_bound})), and
\[\left|\frac{G'(1+it)}{G(1+it)}\right| \leq a_3 t^2 + b_3 |t| + c_3\,,\]
(from (\ref{der-Barnes})), we get
\begin{align}
|f'(t)| \leq a_4 t^2 + b_4 |t| + c_4\end{align}
and then
\begin{align}
\label{upupsilon}
|\Upsilon (i\xi)| \leq a_5 |\xi|^3 + b_5 |\xi|^2 + c_5|\xi|\,.  
\end{align}
More precisely, the following values works well :
\begin{equation}
\label{a5}a_5 = \frac{4}{\beta^2} \ ,  \ b_5 = \frac{4}{\beta^2} + \frac{6}{\beta} \ , \ c_5 = 2 + \frac{3}{\beta} + \frac{2}{\beta^2} \end{equation}
This allows us to get the other bound
\begin{align}
|\Upsilon (i\xi)|  \leq \left(\frac{2}{\b^2}+\frac{3}{\b}+2\right)|\xi|e^{2|\xi|}
\end{align}
Gathering the above estimates 
\[|\Upsilon (i\xi)| e^{|\Upsilon (i\xi)|}\leq \left(\frac{2}{\b^2}+\frac{3}{\b}+2\right)|\xi|\exp (a_5 \xi|^3 + b_5 |\xi|^2 + (c_5 +2)|\xi|)\]
Now,  as in the Hermite case, we can find $k' > 0$ such that for every $\xi \in \mathbb R$
\[a_5 |\xi|^3 + b_5 |\xi|^2 +( c_5+2) |\xi| \leq 2 a_5 |\xi|^3 + k'\,,\]
so that
\begin{align}
\label{interm}|\Upsilon (i\xi)| e^{|\Upsilon (i\xi)|}\leq  C_\b |\xi|e^{\frac{8}{\b^2}|\xi|^3},\end{align}
where $C_\b$ is a constant depending on $\b$ which can be computed explicitly.
It remains to use 
 (\ref{programL}) again to get,
\begin{align*}
\left|\Psi^{L, \beta}(i\xi)-1\right|
&\le C_\b |\xi|e^{\frac{8}{\b^2}|\xi|^3}\,.
\end{align*}
\vskip 10pt
\noindent\underline{2) Bound for  $\left|\psi_n(i\xi)-1\right|$}.
From Theorem \ref{main_theorem}, 
\begin{equation*}
\psi^{L,\b}_n(z)=e^{\Upsilon(z)+ u_n(z)}
\end{equation*}
where
\begin{equation*}
u_n(z) := O\left(\frac{|z|+|z|^2+|z|^3}{n}\right)\,.
\end{equation*}
as soon as $z\in S_{\frac{\b}{2}}$ and $|z|\le\frac{\b}{8}n^{1/6}$.

It is then enough to apply the same trick as in the Hermite case, using (\ref{upupsilon}) and (\ref{interm}).
 We obtain
\[|\psi_n^{L, \beta} - 1| \leq \widetilde C_\beta |\xi| e^{\frac{8}{\beta^2}[\xi|^3}\]

for come constant  $\widetilde{C}_\b$  depending on $\b$ and $\tilde{K}$. Taking
$K_1^\b= \widetilde{C}_\b\}$, $K_2^\b=\frac{8}{\b^2}$ and $D^\b=\frac{1}{4K_2^\b}$, we have have proved that the sequence $\left(X_n^{L,\b}\right)$ converges mod-Gaussian convergence with zone of control $\left[-D^\b t_n^{\b},D^\b t_n^{\b}\right]$ and index of control $(1,3)$. \\
Theorem \ref{determinants-speed} then follows considering $\g=\min\{1,\frac{\nu-1}{2}\}=0$.\\
\\
\bigskip

Note that, in order to be consistent with statement of the theorems, we will not distinguish the constants $D^\b$ for each ensemble. This can be easily achieved by taking the minimum among these constants of $\b$-Laguerre, $\b$-Jacobi,   $\b$-Jacobi Circular and $\b$-Uniform Gram and $\b$-Jacobi Circular ensembles.
\end{proof}

\begin{proof}[Proof of theorems \ref{charactpoly-speed} and \ref{charactpoly-llt}]
In order to avoid repetitions, for the Circular Jacobi ensemble we don't repeat the whole proof, which uses the same technique as for log-determinants, we just remark on the main differences. Solving this case, boils down to finding an upper bound for the function $\Upsilon_\d(i\xi)$, defined in Theorem \ref{main_theorem2}, of the form $\left|\Upsilon_\d(i\xi)\right|\le a|\xi|^3+b|\xi|^2+c|\xi|$. To do so, one can proceed as we did previously for bounding $\Upsilon(i\xi)$, i.e. by applying  \cite[Theorem 5.19]{R76}. This leave us with the problem of estimating from above the following derivative:
\begin{align*}
f'(t)&= i\frac{ G'\left(\frac{i t+\d}{\b'}+1\right)}{G\left(\frac{i t+\d}{\b'}+1\right)} - i\ell \left( \frac{i t+\d}{\b'} +1\right) - \frac{i}{\b'} \left( i t +\d- \half\right) \Psi\left(\frac{i t+\d}{\b'}+1\right).
\end{align*}
The required bound follows by applying again Formulas (\ref{Binet1}), (\ref{digamma}) and (\ref{der-Barnes}) and using successively the inequality 
$$\left|\log(i t+\delta+\b')-\log(\delta+\b')\right|\le |t|\int_0^1\frac{1}{|i ts+\delta+\b'|}ds\le \frac{|t|}{\Re(\delta)+\b'}$$
where $\Re(\d)+\b'>0$, thanks to the assumption of Theorem \ref{main_theorem2}. The latter comes from the representation,
$$\log(a+\zeta)-\log(a)=\int_0^1\frac{\zeta}{s\zeta+a}ds,$$
which is valid for any $a,\zeta\in\C$ as soon as the segment $(a,a+\zeta)$ does not intersect $(-\infty,0)$, and which gives, as $|a+i\xi|\ge a$ for $a,\xi\in\R$,
$$\left|\log(a+i\xi)-\log(a)\right|\le \frac{|\xi|}{a}.$$
One can then apply the estimate for $\left|\Upsilon_\d(i\xi)\right|$ three times to the particular case of the $\b$-Circular Jacobi ensemble and proceed exactly as we did previously.
\end{proof}

\section{Refinements for \texorpdfstring{$\b$}{Lg} rational}\label{section:betarational}
In section \ref{section:results}, we have seen that the term $\Upsilon(z)$ defined in (\ref{upsilon}) appears as part of the limiting function for all the eigenvalues statistics of $\b$-ensembles considered and therefore plays an important role in precise moderate deviations. For this reason, one could ask if it is possible to obtain a simpler expression for $\Upsilon(z)$ not depending on the integral $\int_0^\infty\varphi(s) \frac{\ex^{-sz} - 1}{\ex^{s \b'} - 1} ds$, which does not have a closed formula for all $\b > 0.$ \\
The subsequent lemma shows that, when $\b$ is a positive half integer or rational, $\Upsilon(z)$  can be written as a finite sum of $\log$-Gamma and $\log$-Barnes G-functions. This is done via a direct method which uses the properties of such special functions and meets computations similar to those of Corollary 3.2 in \cite{ES97}. See also \cite{BG11} where similar expansions are performed for $z=1$.\\
Note that, the ausiliary study of the case $\b\in\Q$ implies lemma \ref{GUE-lemma}. Additionally, it allows to recover the limiting function of the COE, CUE and CSE characteristic polynomial previously computed by Kowalski and Nikeghbali in \cite{KN12}.
\begin{lemma} \label{rational}
For all $n\ge 1$ and any $z\in S_{\frac{\b}{2}}$ with $|z| <\frac{\b}{8} n^{1/6}$, 
\begin{enumerate}
\item\label{integer}
if $\frac{\b}{2}\in\N$, then theorem \ref{main_theorem} holds replacing $\Upsilon$ with 
\begin{align*}
\Upsilon^{\N} (z) &= \frac{z}{\b} \log{2\pi} + \frac{z^2}{\b} \log \frac{\b}{2}- \frac{2}{\b} \log G(1+z)  \\
&\quad +\sum_{m=1}^{\frac{\b}{2}-1} \frac{2m}{\b} \left(\ell
\left(\frac{2m}{\b} \right) - \ell
\left( \frac{2m+2z}{\b}\right)\right);\qquad\qquad\qquad\qquad
\end{align*}
\item\label{oneratio}
if $\frac{\b}{2}= \frac{1}{q},\  q\in \N$, then theorem \ref{main_theorem} holds replacing $\Upsilon$ with 
\begin{align*}
\Upsilon^{\frac{1}{\N}}(z) &= z \left( \left(\half-\frac{q}{2}\right)\log\frac{1}{q} + \frac{1}{2} \log{2\pi}\right) - \frac{1}{q} \log G\left(1+qz\right)\\
&\quad+ \sum_{m=1}^{q-1} \left(\frac{m}{q}-1\right) \left(\ell\left(\frac{m}{q}\right) - \ell \left( \frac{m}{q}+z\right)\right);\qquad\qquad\qquad
\end{align*}
\item\label{pration}
if $\frac{\b}{2}=\frac{p}{q},\ p, q \in \N$, then theorem \ref{main_theorem} holds replacing $\Upsilon$ with 
\begin{align*}
\Upsilon^{\Q}(z) &= z\left(\left(\half-\frac{q}{2}\right)\log\frac{1}{q}+\half\log 2\pi\right)\\
&\quad -\sum_{l=0}^{p-1} \frac{1}{q} \left( \log G\left(1+\frac{z+l}{p}q\right) - \log G\left(1+\frac{l}{p}q\right)\right)\\
&\quad + \sum_{l=0}^{p-1} \sum_{m=1}^{q-1} \left(\frac{m}{q}-1\right) \left(\ell\left( \frac{m}{q}+\frac{l}{p}\right) - \ell\left( \frac{m}{q}+\frac{z+l}{p}\right)\right).
\end{align*}
\end{enumerate}
\end{lemma}
\begin{proof} 
\textit{Part 1.}
Again we set $\b'=\bh$. Our proof is similar to the one of Corollary 3.2 in \cite{ES97}, where the case of $\b' = 1$ is examined. We start considering the following product
\begin{equation}\label{prod1}
\prod_{k=1}^n \frac{\G(\b' k+z) }{\G(\b^\prime k)\G(1+z)}.
\end{equation}
Using the Euler expansion (\ref{gamma})  of the Gamma function
we obtain  the following representation for (\ref{prod1}),
\begin{align*}
\prod_{k=1}^n \frac{\G(\b' k+z) }{\G(\beta^\prime k)\G(1+z)}
&=\prod_{k=1}^n\prod_{j=1}^\infty \frac{j}{k+\b'k-1+z}\frac{(j+\b'k-1)(j+z)}{j^2}\\
&=\prod_{j=1}^\infty\left(1+\frac{z}{j}\right)^n\prod_{k=1}^n \left(1+\frac{z}{j+\b'k-1}\right)^{-1}
=\prod_{j=1}^{\b'n-1}\left(1+\frac{z}{j}\right)^{n-\lfloor\frac{j}{\b'} \rfloor},
\end{align*}
where the last equality comes from a simple counting argument. Indeed, let $l$ be a fixed positive integer. If $l \ge \b' n$, i.e.  $l=\b'n+p$ for some $p\ge 0$, there are $n$ terms like $(1 + zl^{-1})^{-1}$ coming from the product over $k$. Indeed there are $n$ couples $(j,k)$, with $j\ge 1$ and $k=1,\ldots,n$ realizing the equality $j+\b'k-1=\b' n+p = l$. \\
If $l< \b'n$, we have $\lfloor\frac{l}{\b'} \rfloor$ those terms coming from the product over $k$: since $\b'\lfloor\frac{l}{\b'} \rfloor\le l< \b'\left(\lfloor\frac{l}{\b'} \rfloor+1\right)$, it's easy to check that there are $\lfloor\frac{l}{\b'} \rfloor$ couples $(j,k)$, with $j\ge 1 $ and $k=1,\ldots,n$ realizing the equality $j+\b'k-1=l$.\\ 
Thus,
\begin{align*}
\prod_{k=1}^n \frac{\G\left( \b'k +z \right)}{\G \left( \b'k\right)}&=S_1(n,z)S_2(n,z),
\end{align*}
where
\begin{equation*}
S_1(n,z):=\G^n\left( 1 +z \right) \prod_{j=1}^{\b' n} \left( 1 + \frac{z}{j}\right)^{n -  \frac{j}{\b'}  }
\end{equation*}
and
\begin{equation*}
S_2(n,z):=\prod_{j=1}^{\b' n} \left( 1 + \frac{z}{j}\right)^{  \frac{j}{\b'} - \lfloor \frac{j}{\b'} \rfloor }.
\end{equation*}
Using the expansion (\ref{expG}) of  G  
together with (\ref{gamma}), we obtain
\begin{equation}
\label{s1} 
S_1(n,z) = \frac{\left( 2 \pi \right)^{z/{2\b'}} e^{-\frac{z}{\b'}}}{\left(G(1+z)\right)^{\frac{1}{\b'}}}  \left( \b'n+1\right)^{\frac{z^2}{2\b'}+zn} e^{-zn} S_3(n,z) 
\end{equation}
where
\begin{equation*}
S_3(n,z):= e^{-\frac{z(z-1)}{2\b'}}\prod_{j=\b' n+1}^\infty  e^{-\frac{z}{\b'}}  \left( 1+\frac{z}{j}\right)^{\frac{j}{\b'} - n} \left(1 + \frac{1}{j}\right)^{\frac{z^2}{2\b'}+zn}.
\end{equation*}
We start showing that
\begin{equation}\label{s3}
\log S_3(n,z)=O\left(\frac{|z|+|z|^2+|z|^3}{n}\right)
\end{equation}
with an absolute implied constant. If we Taylor-expand the logarithm at the origin we have that
\begin{equation*}
\log (1+\o)=\o-\frac{\o^2}{2}+O(\o^3)
\end{equation*}
for $|\o|\le a<1$, with an absolute implied constant. 
Since the sum in $\log S_3(n,z)$ runs from $\b'n+1$ and  $|z|<\frac{\b'}{4}n^{1/6}$, both terms $\left|\frac{z}{j}\right|$ and $\left|\frac{1}{j}\right|$ are bounded by a constant $a<1$. Therefore we get
\begin{align*}
\log S_3(n,z) &=  -\frac{z(z-1)}{2\b'} + \sum_{j>\b' n}^\infty \left[  -\frac{z}{\b'} + \frac{j}{\b'} \log\left( 1+\frac{z}{j}\right) +  \frac{z^2}{2\b'} \log \left(1 + \frac{1}{j}\right) \right] \\
& \quad + \sum_{j>\b' n}^\infty \left[ zn \log \left( 1+\frac{1}{j}\right) - n \log\left(1 + \frac{z}{j}\right) \right] \\
&= -\frac{z(z-1)}{2\b'} -\frac{z^2}{4\b'}\sum_{j>\b' n}^\infty\frac{1}{j^2}+ \frac{z(z - 1)}{2}\sum_{j>\b' n}^\infty \left( \frac{n }{j^2}\right)\\
& \quad + O\left(\sum_{j>\b' n}^\infty\left(\frac{nz}{k^3}+\frac{z^2}{k^3}+\frac{z^3}{k^2}\right)\right)
\end{align*}
with an absolute implied constant. Using that 
\begin{equation*}
\sum_{k>n}\frac{1}{j^2}=\frac{1}{n}+O\left(\frac{1}{n^2}\right),\qquad \sum_{k>n}\frac{1}{j^3}=\frac{1}{2n^2}+O\left(\frac{1}{n^3}\right)
\end{equation*}
for all $n\ge 1$, (\ref{s3}) follows.
Parameterizing in terms of the reminders of $j/\b'$, we derive the following representation for $S_2(n,z)$,
\begin{align}\label{s2b}
\nonumber S_2(n,z)&= \prod_{j=1}^{\b' n} \left( 1 + \frac{z}{j}\right)^{  \frac{j}{\b'} - \lfloor \frac{j}{\b'} \rfloor } = \prod_{m=1}^{\b'-1} \prod_{p=0}^{n-1} \left( 1 + \frac{z}{p\b'+m}\right)^{\frac{m}{\b'}} =\\
\nonumber &= \prod_{m=1}^{\b'-1} \left(\frac{m+z}{m} \frac{\G\left(\frac{m+z}{\b'}+n \right)}{\G\left(\frac{m}{\b'}+n \right)}\frac{\G\left(\frac{m}{\b'}+1 \right)}{\G\left(\frac{m+z}{\b'}+1 \right)} \right)^{\frac{m}{\b'}}\\
&= \prod_{m=1}^{\b'-1} \left(\frac{\G\left(\frac{m+z}{\b'}+n \right)}{\G\left(\frac{m}{\b'}+n \right)}\frac{\G\left(\frac{m}{\b'}\right)}{\G\left(\frac{m+z}{\b'}\right)} \right)^{\frac{m}{\b'}},
\end{align}
where the second inequality can be shown using (\ref{gamma}).\\
Gathering the estimates (\ref{s1}), (\ref{s3}) and (\ref{s2b}), we arrive at 
\begin{align*}
\sum_{k=1}^n \ell\left( \b'k+z\right) - \ell\left( \b'k\right)&= \frac{z}{2\b'} \log{2\pi}-\frac{z}{\b'} - \frac{1}{\b'} \log G(1+z)  \\
&\quad +\left( \frac{z^2}{2\b'}+zn\right) \log\left( \b'n+1\right)-zn+\log(S_n(z)) \\
&\quad +\sum_{m=1}^{\b'-1} \frac{m}{\b'} \left(\ell\left(\frac{m}{\b'} \right) - \ell\left( \frac{m+z}{\b'}\right)\right) \\
&\quad +S_4(n,z)+O\left( \frac{|z|+|z|^2+|z|^3}{n}\right),
\end{align*}
where 
\begin{equation*}
S_4(n,z):=\sum_{m=1}^{\b'-1} \frac{m}{\b'} \left( \ell\left( \frac{m+z}{\b'}+n\right) - \ell\left( \frac{m}{\b'}+n\right)\right).
\end{equation*}
Note that, again using the Binet's formula (\ref{Binet1}),
\begin{align*}
S_4(n,z)&=\sum_{m=1}^{\b'-1} \frac{m}{\b'}  \left( \frac{m+z}{\b'} +n - \half \right) \log\left( 1+\frac{z}{m+n\b'}\right) \\
&\quad +\sum_{m=1}^{\b'-1} \frac{mz}{\left(\b'\right)^2} \log\left( \frac{m}{\b'}+n\right) - \sum_{m=1}^{\b'-1} \frac{mz}{\left(\b'\right)^2} \\
&\quad + \sum_{m=1}^{\b'-1}\frac{m}{\b'}\int_0^\infty f(s)e^{-sn}\left(e^{-\frac{s}{\b'}}\right)^m\left(e^{-\frac{sz}{\b'}}-1\right)ds\\
&= \sum_{m=1}^{\b'-1} \frac{m}{\b'}  n \log\left( 1+\frac{z}{m+n\b'}\right) + \left(\log n-1\right)\sum_{m=1}^{\b'-1} \frac{mz}{\left(\b'\right)^2} +O\left(\frac{|z|}{n}\right)\\
&= \frac{\b'-1}{2\b'} z \log n+O\left(\frac{|z|+|z|^2}{n}\right),
\end{align*}
so that 
\begin{equation}\label{s4}
S_4(n,z)=\frac{\b'-1}{2\b'} z \log n+O\left(\frac{|z|+|z|^2}{n}\right)
\end{equation}
with an absolute implied constant. The result then follows joining the estimates derived above.
\vskip 20pt
\textit{Part 2.} Without loss of generality we assume $q|n$.
Recall the multiplication theorem for the Gamma function, 
\begin{align}
\prod_{k=0}^{m-1} \G\left(z+\frac{k}{m} \right) = \left(2 \pi \right) ^{\frac{m-1}{2}} m^{\half - mz} \G\left( m z \right). \label{mult}
\end{align}
Note that the product we are interested in, may be rewritten as follows,
\begin{align*}
\prod_{k=1}^n \frac{\Gamma\left(\frac{k}{q}+z\right) }{\Gamma\left(\frac{k}{q}\right)}&=\frac{\Gamma\left(\frac{n}{q}+z\right) }{\Gamma\left(\frac{n}{q}\right)}
\prod_{k=1}^{q-1} \frac{\Gamma\left(\frac{k}{q}+z\right) }{\Gamma\left(\frac{k}{q}\right)}
\ \prod_{k=q}^{n-1} \frac{\Gamma\left(\frac{k}{q}+z\right) }{\Gamma\left(\frac{k}{q}\right)}\\
&=\frac{\Gamma\left(\frac{n}{q}+z\right) }{\Gamma\left(\frac{n}{q}\right)}\ \prod_{k=1}^{q-1} \frac{\Gamma\left(\frac{k}{q}+z\right) }{\Gamma\left(\frac{k}{q}\right)}
\ \prod_{j=1}^{\frac{n}{q}-1}\prod_{r=0}^{q-1}\frac{\G\left(j+\frac{r}{q}+z\right)}{\G\left(j+\frac{r}{q}\right)}\\
&=q^{-zn+qz}\frac{\Gamma\left(\frac{n}{q}+z\right) }{\Gamma\left(\frac{n}{q}\right)}\ \prod_{k=1}^{q-1} \frac{\Gamma\left(\frac{k}{q}+z\right) }{\Gamma\left(\frac{k}{q}\right)}\prod_{j=1}^{\frac{n}{q}-1}\frac{\G(qj+qz)}{\G(qj)}
\end{align*}
where the last equality is due to (\ref{mult}).
The theorem then follows by a simple application of Lemma \ref{rational}.\ref{integer} and of the first Binet formula (\ref{Binet1}).
\vskip 20pt
\textit{Part 3.} Again we assume,  without loss of generality, that $q|n$.
Applying the multiplication formula (\ref{mult}),
\begin{align*}
\prod_{k=1}^n \frac{\G\left(\frac{p}{q} k+z\right) }{\G \left(\frac{p}{q} k\right)}& = \prod_{k=1}^n \frac{\G\left(p \left(\frac{k}{q} +\frac{z}{p}\right)\right) }{\G \left(\frac{p}{q} k\right)} = p ^{zn}\prod_{k=1}^n \prod_{l=0}^{p-1} \frac{\G\left(\frac{k}{q} +\frac{z}{p} + \frac{l}{p}\right) }{\G \left(\frac{k}{q} +\frac{l}{p}\right)}\\
&=  p^{nz}\prod_{l=0}^{p-1}\left(\prod_{k=1}^n\frac{\G\left(\frac{k}{q}+\frac{z+l}{p}\right)}{\G\left(\frac{k}{q}\right)}\right)\left(\prod_{k=1}^n\frac{\G\left(\frac{k}{q}+\frac{l}{p}\right)}{\G\left(\frac{k}{q}\right)}\right)^{-1}.
\end{align*}
Then the result follows immediately by applying twice Lemma \ref{rational}.\ref{oneratio}. 
\end{proof}

\section{Appendix}
\label{appendix}
\subsection{Some properties of the Gamma and Barnes G functions}
\textit{The Gamma function.} For any $z\in\C$ with $\Re z>0$, the complex Gamma function is defined as the absolute convergent integral
\begin{equation*} 
\G(z) = \int_0^\infty \ex^{-t} t^{z-1} dt.
\end{equation*} 
Its extension as a meromorphic function on $\mathbb C \setminus \{-1, -2, \dots\}$ has the representation
\begin{equation}\label{gamma}
\G\left( 1+z\right) = \prod_{j=1}^\infty \Big\{ \left( 1+\frac{z}{j}\right)^{-1} \left( 1+\frac{1}{j}\right)^z \Big\}.
\end{equation} 
The first Binet's formula for the logarithm of the Gamma function $\ell (z):=\log\G(z)$ is given by  (see for instance \cite{EMOT} p.21 or \cite{WW96} p.242)
\begin{equation}
\label{Binet1}
\ell (z) = \left( z - \half\right) \log z - z + 1 + \int_0^\infty \varphi(s) \left[ \ex^{-sz} - \ex^{-s}\right]ds,  \quad \Re (z) >0.
\end{equation}
where the function $\varphi$ is defined as
\begin{equation*}
\varphi(s) = \left[\half - \frac{1}{s} + \frac{1}{\ex^s-1} \right] \frac{1}{s}
\end{equation*}
and satisfies, for every $s\ge 0$
\begin{equation}\label{12}
0<\varphi(s)\le \varphi(0)=\frac{1}{12}.
\end{equation}
Moreover, if we set
\begin{equation*}
\phi(s) = \frac{1}{12}-\varphi(s) \ , \quad s >0
\end{equation*}
we have that  $0 \leq  \phi(s) < 1/12$ for $s > 0$ and
\begin{equation*}
\lim_{x \rightarrow 0} \phi(x)/x^2 = 1/720,
\end{equation*}
which leads to another version of (\ref{Binet1}), that is
\begin{equation}
\label{Binet1bis}
\ell (z) = \left(z - \half\right)\log z -z + \half \log 2\pi + \frac{1}{12z} - \int_0^\infty \phi(s) \ex^{-sz} ds\,.
\end{equation}
The second Binet's formula (see for instance \cite{EMOT} p.22 or \cite{WW96} p.245) is
\begin{equation}
\label{Binet2}
\ell (z) = \left(z-\half \right) \log z - z + \frac{1}{2} \log (2 \pi) + 2 \int_0^\infty \frac{ \arctan\left(\frac{s}{z} \right)}{\ex^{2 \pi s} - 1} ds, \quad \Re (z) >0,
\end{equation}
where, for complex $\zeta$, $\arctan\zeta$ is defined as
\[\arctan \zeta := \int_0^\zeta \frac{dt}{1+t^2}\]
with integration along a straight line.
Both have as a byproduct the classical Stirling-like formula (see Olver \cite{FR97} p.293 or \cite{WW96} p.243)
\begin{equation}
\label{Stirling}
\ell (1+z)  = \left( z + \half \right) \log z - z + \half \log 2 \pi + \frac{1}{12 z} +  r(z), 
\end{equation}
where, for every $\delta > 0$, 
\begin{equation*}
\sup_{|\arg z| \leq \pi-\delta}|z|^3 |r(z)| = R_\delta < \infty.
\end{equation*}
The derivative of the logarithm of the Gamma function $\Psi(z)$, called the Digamma function, can be represented, differentiating (\ref{Binet1}), as 
\begin{equation}\label{digamma}
\Psi(z) = \log(z) - \int_0^\infty \ex^{-sz} \left( s \varphi(s) + \half\right)ds.
\end{equation}
Moreover, 
\begin{align}
0 < s\varphi(s) + \half < 1. \label{digamma_bound}
\end{align}
\textit{The Barnes G function.}  It is defined as the entire solution of the functional equation
\begin{equation}
\label{finiteG}
G(z+1) = G(z) \G (z),
\end{equation}
which may be represented as the product
\begin{equation}
\label{expG}
G \left( 1+z \right) = \left( 2 \pi \right)^{z/2} e^{-(z+1)z/2} \prod_{j=1}^\infty \Big\{ \left( 1+\frac{z}{j}\right)^j \left(1 + \frac{1}{j}\right)^{z^2/2} \ex^{-z}\Big\}\,.
\end{equation}
Its derivative satisfies (see \cite{WW96} p.258)
\begin{equation}\label{der-Barnes}
\frac{G'(z)}{G(z)} = (z-1) \Psi(z) - z + \half \log 2\pi + \half
\end{equation}
and its logarithm is related to the $\ell$ function as follows (see \cite{Barnes})
\begin{equation}
\label{Barnes}
\log G(z+1) = \frac{z(1-z)}{2} + \frac{z}{2} \log (2 \pi) + z \ell(z) - \int_0^z \ell(x) dx 
, \quad \Re (z) >0.
\end{equation}
Gathering the expansions (\ref{Barnes}) and (\ref{Binet2}), we obtain the following integral representation for all $z\in\C$ with $\Re (z)>0$,
\begin{align*}
\log G(z+1)  
&= \frac{z}{2} - \frac{3z^2}{2} + z \log (2 \pi) + z \left(z-\half \right) \log z  -  \frac{z}{2} \left( 1-\frac{z}{2} + (z-1) \log z \right) \\
&\quad + \frac{z^2}{2} - \frac{z}{2} \log (2 \pi) - 2 \int_0^\infty \frac{  \frac{1}{2} \log \left(z^2+s^2 \right) - s \log s}{\ex^{2 \pi s} - 1} ds ,
\end{align*}
which gives 
\begin{equation}
\label{logG}
\log G(z+1) = \frac{z^2}{2} \log z -\frac{3}{4}z^2 + \frac{z}{2} \log 2 \pi   - \int_0^\infty   \log \left(1+ z^2s^{-2} \right) \frac{s\!\ ds }{\ex^{2 \pi s} - 1} .
\end{equation}
Moreover, there is an asymptotic formula (Formula 4.184 in \cite{F10})
\begin{equation}
\label{4.184}
\log G(z+1) =  \frac{z^2}{2}\log z - \frac{3}{4}z^2
 + \frac{z}{2}\log 2\pi - \frac{1}{12}\log z + \zeta'(-1) + o(1) \end{equation}
 where $\zeta$ is the Riemann function. 

\subsection{The Abel-Plana summation formula }\begin{thm}[Abel-Plana Summation Formula, \cite{FR97}, p. 290] \ 
\label{Abel-Plana}\\
Let $f$ be a holomorphic function on the
strip $\lbrace z \in \C, 0 \leq \Re(z) \leq n \rbrace$ (i.e. $f$ is continuous on this strip and holomorphic in
its interior). Suppose that $f(z) = o \left(e^{2 \pi |\Im(z)|}\right)$ as $\Im(z) \to \pm \infty$, uniformly
with respect to $\Re(z) \in [0, n].$ Then,
\begin{align}
\nonumber
 \sum_{k=0}^{n-1} f(k) &= \int_{0}^{n} f(s) ds  + \half f(0) - \half f(n) +i \int_{0}^\infty \frac{f(is) - f(-is)}{\ex^{2 \pi s} - 1} ds \\
\label{AP}
&- i\int_0^\infty \frac{f(n + is) - f(n - is)}{ \ex^{2 \pi s} -1} ds.
\end{align}
\end{thm}
\bibliographystyle{plain}
\bibliography{Biblio}
\end{document}